\documentclass[preprint,sort&compress,12pt]{elsarticle}
\usepackage{bbm}
\usepackage{amstext}
\usepackage{amsmath}
\usepackage{amssymb}
\usepackage{mathrsfs}
\usepackage{stmaryrd}
\usepackage{bm}
\usepackage{pstricks}
\usepackage{pst-coil}
\usepackage{pst-3d}
\usepackage{amsfonts}
\usepackage{amsthm}
\usepackage{CJK}
\allowdisplaybreaks

\newcommand{\mm}{\mathrm}
\newcommand{\ml}{\mathcal}

\newcommand{\be}{\begin{equation}}
\newcommand{\bea}{\begin{equation}\begin{aligned}}
		\newcommand{\beas}{\begin{equation*}\begin{aligned}}
				\newcommand{\eeas}{\end{aligned}\end{equation*}}
		\newcommand{\eea}{\end{aligned}\end{equation}}
\newcommand{\ee}{\end{equation}}

\renewcommand{\div}{{\rm div }}

\begin{document}
\begin{CJK*}{GBK}{song}
	\begin{frontmatter}
		\title{On Temporal Decay of Compressible  Hookean Viscoelastic Fluids \\ with
			Relatively Large Elasticity Coefficient}
		
		\author[mml1]{Shengbin Fu}
		\ead{fsb20249100@scut.edu.cn}
		\address[mml1]{School of Mathematics, South China University of Technology, Guangzhou, 510641, P.R. China.}
		\author[mml2]{Wenting Huang}
		\ead{hwting702@163.com}
		\address[mml2]{School of Mathematical Sciences, Beijing
			Normal University, Beijing 100875, China.}
		\author[FJ]{Fei Jiang\corref{cor1}}\cortext[cor1]{Corresponding author.}
		\ead{jiangfei0591@163.com}
		\address[FJ]{School of Mathematics and Statistics, Fuzhou University, Fuzhou, 350108, China.}
		\begin{abstract}
			Recently, Jiang--Jiang (J. Differential Equations 282, 2021) showed the existence of unique strong solutions in spatial periodic domain (denoted by $\mathbb{T}^3$), whenever the elasticity coefficient is  larger than the initial velocity perturbation of the rest state. Motivated by Jiang--Jiang's result, we revisit the Cauchy problem of the compressible viscoelastic fluids in Lagrangian coordinates. Employing an energy method with temporal weights and an additional asymptotic stability condition of initial density in Lagrangian coordinates, we extend the Jiang--Jiang's result with exponential decay-in-time in $\mathbb{T}^3$ to the one with algebraic decay-in-time in the whole space $\mathbb{R}^3$. Thanks to the algebraic decay of solutions established by the energy method with temporal weights, we can further use the spectral analysis to improve the temporal decay rate of solutions. In particular, we find that the $k$-th order spatial derivatives of both the density and deformation perturbations converge to zero in $L^2(\mathbb{R}^3)$   at a rate of $(1+t)^{-\frac{3}{4}-\frac{k+1}{2}}$, which is faster than the decay rate $(1 +t)^{-\frac{3}{4}-\frac{k}{2}}$ obtained by Hu--Wu  (SIAM J. Math. Anal. 45, 2013) for $k=0$ and $ 1$. In addition, it's well-known that the decay rate  $(1+t)^{-\frac{3}{4}-\frac{k}{2}}$ of the density perturbation  is optimal in the compressible Navier--Stokes equations (A.~Matsumura, T.~Nishida, Proc. Jpn. Acad. Ser-A. 55, 1979). Therefore, our faster temporal decay rates indicate that the elasticity accelerates the decay of the density perturbation  after the rest state of a compressible viscoelastic fluid being perturbed.
		\end{abstract}
	\end{frontmatter}
	
	
	\newtheorem{thm}{Theorem}[section]
	\newtheorem{lem}{Lemma}[section]
	\newtheorem{pro}{Proposition}[section]
	\newtheorem{concl}{Conclusion}[section]
	\newtheorem{cor}{Corollary}[section]
	\newproof{pf}{Proof}
	\newdefinition{rem}{Remark}[section]
	\newtheorem{definition}{Definition}[section]

	\section{Introduction}\label{introud}
	\numberwithin{equation}{section}
	
	Viscoelastic materials include a wide range of fluids with elastic properties, as
	well as solids with fluid properties. The models of viscoelastic fluids formulated by Oldroyd, in particular the classical Oldroyd model,
	have been studied by many authors. In this paper, we consider the following compressible Oldroyd model that is defined on $\mathbb{R}^3$, and includes a viscous stress component
	and a stress component for a neo-Hookean solid
	\cite{MR2922368}:
	\begin{equation}
		\label{1.1}
		\begin{cases}
			\rho_t+\mm{div}(\rho v)=0,\\
			\rho v_t+\rho v\cdot\nabla v+\nabla P(\rho)= \mu\Delta v
			+\lambda \nabla \mm{div}v +  \kappa\mm{div}(
			{FF^{\top}}/{\det F}), \\
			F_t+ v\cdot\nabla F =\nabla v F.
		\end{cases}
	\end{equation}
	Here the unknowns $\rho(x,t)$, $v(x,t)$ and $F(x,t)$ denote the density, the velocity and the
	deformation gradient (a $3\times 3$ matrix
	valued function), respectively. The hydrodynamic
	pressure $P(\rho) \in C^2(\mathbb{R})$ is an increasing and
	convex function for $\rho > 0$. Besides, the viscosity
	coefficients $\mu $ and $\lambda $ satisfy the strongly
	elliptic conditions: $ \mu >0$ and $3\lambda +2\mu >0$. For the
	convenience, the elasticity coefficient $\kappa$ is assumed
	to be a positive constant, and the term $\kappa\mm{div}(
	FF^{\top}/\det F)$ is referred as the elasticity.
	
	The local existence of strong solutions of the initial value problem of the system $(1.1)$ was shown by Hu--Wang \cite{MR2652169}. The global(-in-time) existence of  strong solutions of the initial value problem of the system $(1.1)$ was proved by Hu--Wang \cite{MR2737829}, Qian--Zhang \cite{MR2729321}  and Hu--Wu \cite{MR3101093}, provided that the initial perturbation  $\left(\rho^{0}-1, v^{0}, F^{0}-I\right)$  is sufficiently small, where the steady state of the system \eqref{1.1} is $(1,0,I)$ with  $3\times 3$ identity matrix $I$. Hu--Wu \cite{MR3101093} also showed that if the initial perturbation  $\left(\rho^{0}-1, v^{0}, F^{0}-I\right)$  belongs to  $L^{1}\left(\mathbb{R}^{3}\right) \cap H^{2}\left(\mathbb{R}^{3}\right)$, then it holds for any $t \geqslant 0 $ that
	\begin{align}
		& \|\nabla^k(\rho(t) -1, v(t), F(t)-I)\|_{L^2(\mathbb{R}^3)}\leqslant C(1 +t)^{-\frac{3}{4}-\frac{k}{2}},\ k=0,\ 1, \label{2022410252025}
	\end{align} via the Fourier splitting method and the Hodge decomposition. Here and in what follows, the letter $C$ denotes the positive constant, which depends on the given initial data for the system \eqref{1.1}. In particular, using the interpolation inequality,  the  $L^{p}$-norm ($2\leqslant p\leqslant 6$) decay estimates are easily further obtained from \eqref{2022410252025}. Li--Wei--Yao \cite{MR3576321,MR3538871} extended the above result to the case  $2 \leqslant p \leqslant \infty$, and obtained the decay estimates in  $L^{2}$  of higher-order spacial derivatives:
	\begin{equation*}
		\left\|\nabla^{k} \left(\rho(t)-1, v(t), F(t)-I \right)\right\|_{L^{2}(\mathbb{R}^3)} \leqslant C(1+t)^{-\frac{3}{4}-\frac{k}{2}},\ k=0,\ 1,\ldots,\ N-1,
	\end{equation*}
	provided that the initial data satisfy that $\left(\rho^{0}-1, v^{0}, F^{0}-I\right)$  belongs to  $H^{N}$ with $N \geqslant 3$, and is small in  $L^{1} \cap H^{3}$. This follows from the diffusive aspect of the system $(1.1)$. We also refer to  \cite{MR3812223,MR3584369,MR3927502,MR3632688,MR4435915,MR4350195} in relevant progresses and \cite{MR4335131,MR3909068,MR4059368}  for the decay of solutions to the other closely related models in viscoelastic fluids.  Later on, Ishigaki \cite{MR4152207} further showed
	that if the initial perturbation  $\left(\rho^{0}-1, v^{0}, F^{0}-I\right)$  is sufficiently small in  $L^{1} \cap H^{3}$, then the global classical solution satisfies the following  $L^{p}$  decay estimate:
	\begin{equation*}
		\left\|(\rho(t)-1, v(t), F(t)-I)\right\|_{L^{p}(\mathbb{R}^3)} \leqslant C(1+t)^{-\frac{3}{2}\left(1-\frac{1}{p}\right)-\frac{1}{2}\left(1-\frac{2}{p}\right)},\ 1<p\leqslant \infty.
	\end{equation*}
	The above result improves the decay rate of the  $L^{p}$  norm of $v$  obtained in \cite{MR3101093,MR3576321} for  $p>2$, and clarifies the diffusion wave phenomena caused by the sound wave and the elastic shear wave in viscoelastic fluids.
	
	Recently Jiang--Jiang \cite{MR4217912} further proved the existence of unique strong solutions in spatial periodic domain (denoted by $\mathbb{T}^3$), whenever the elasticity coefficient is relatively larger than initial velocity perturbation of the rest state. In addition, they further established an interesting result that the system \eqref{1.1}  in Lagrangian coordinates can be approximately by a linear system for sufficiently large $\kappa$. Motivated by Jiang--Jiang's results, we revisit the Cauchy problem of the compressible viscoelastic flows in Lagrangian coordinates. Employing an energy method with temporal weights  and an asymptotic stability condition of initial density in  Lagrangian coordinates, we extend the Jiang--Jiang's result  with exponential decay-in-time  in $\mathbb{T}^3$ to the one with  algebraic decay-in-time  in the whole space $\mathbb{R}^3$, see Theorems \ref{1.12} and \ref{1.10}. Thanks to the algebraic decay-in-time of solutions established by the energy method with temporal weights, we can further use the spectral analysis to improve the temporal decay rate of solutions. In particular, we have
	\begin{align}
		& \|\nabla^k (\rho(t) - 1, F(t)-I)\|_{L^2(\mathbb{R}^3)}\leqslant C(1 +t)^{-\frac{3}{4} -\frac{k+1}{2}},\ k=0,\ 1. \label{1.27}
	\end{align}
	The above  decay rate obviously is faster than the one of $\|\nabla^k(\rho(t) -1, F-I)\|_{L^2(\mathbb{R}^3)}$ in  \eqref{2022410252025}  obtained by Hu--Wu. In addition, it's well-known that the decay rate  $(1 +t)^{-\frac{3}{4}-\frac{k}{2}}$ of the  density perturbation  is optimal in the compressible Navier--Stokes equations  \cite{MANTTP351}.  Therefore our faster temporal decay rates present that the elasticity accelerates the decay of perturbation of density after the rest state of a compressible viscoelastic fluid being perturbed. In addition, we remark that the decay result in Theorem \ref{1.24} can be further extended to $L^p$ with $p>2$ by following Ishigaki's argument in \cite{MR4152207}.
	
	Finally, we mention that  the other mathematical topics for viscoelastic fluids
	have also been widely investigated, such as the incompressible limit \cite{MR2191777}, the global existence of weak solutions \cite{MR3812223},
	the corresponding incompressible case \cite{MR3434615,MR2273974,MR2712454,MR2393434,MR2165379,MR2383932},
	the  global existence of  classical solutions   with small initial perturbations for  the incompressible inviscid case \cite{MR4451474,MR3552009,MR2142629,MR2358646,MR3391913,MR3626302}.
	and so on.   We also refer the readers to \cite{MR4064197,MR3010381,MR1763488,MR2990047} and the references cited therein for the local/global existence of solutions to the other closely
	related models in viscoelastic fluids.
	
	\subsection{Reformulation of the system \eqref{1.1} in Lagrangian coordinates}
	
	Now we reformulate the motion equations \eqref{1.1} in Lagrangian coordinates.
	Let the flow map
	$\zeta$ be the solution to
	\begin{equation}
		\label{1.3}
		\begin{cases}
			\displaystyle\frac{\mm{d}}{\mm{d}t} \zeta(y,t)= v(\zeta(y,t),t),
			\\
			\zeta(y,0)=\zeta^0(y),
		\end{cases}
	\end{equation}
	where $\zeta^0(y): \mathbb{R}^3 \rightarrow \mathbb{R}^3$
	is an invertible map satisfying $\det (\nabla \zeta^0)
	\neq0$.
	Then, the deformation
	tensor $\tilde{F}(y, t)$ in Lagrangian coordinates  takes the form
	\begin{align}\label{1.7}
		\tilde{F}(y,t)=\nabla \zeta(y,t),\ \mbox{i.e.},\ \tilde{F}_{ij}=\partial_{j}\zeta_i(y,t) .
	\end{align}
	Thus the deformation tensor in Eulerian coordinates is given by the formula
	$$ F(x,t)=\nabla\zeta(\zeta^{-1}(x,t),t). $$
	Using the chain rule, we easily check that  $F(x, t)$ in \eqref{1.7}
	automatically satisfies  $\eqref{1.1}_3$. Hence, for the initial
	data $F^0$ enjoys the form
	\begin{align}
		\label{1.6}
		F^0=\nabla \zeta^0(\zeta^{-1}_0,0),
	\end{align}
	the solution $\tilde{F}(y,t)$ can be explicitly expressed
	by \eqref{1.7} in Lagrangian coordinates. Here and in what follows, we also use the notation $f_0$, as well as $f^0$, to represent the initial data of $f$.
	
	Before further rewriting the elasticity $\div(FF^\top)$ in  Lagrangian coordinates, we shall introduce some differential operators. Let $\mathcal{A} =(\mathcal{A}_{ij})_{3 \times 3} = (\partial_j \zeta_i)_{3 \times 3}^{-\top}$. Then, for
	any given scalar function $f$ and for any given vector function
	$X:=(X_1,X_2,X_3)^{\top}$, the differential operators
	$\nabla_\mathcal{A}$,
	$\mm{div}_\ml{A}$ and $\Delta_{\ml{A}}$ are defined by
	$\nabla_{\ml{A}}f:=(\ml{A}_{1k}\partial_k f,
	\ml{A}_{2k}\partial_kf,\ml{A}_{3k}\partial_kf)^{\top}$,
	$\mm{div}_{\ml{A}}(X_1,X_2,X_3)^{\top}:=\ml{A}_{lk}\partial_k
	X_l$
	and
	$\Delta_{\mathcal{A}}f:=\mm{div}_{\ml{A}}\nabla_{\ml{A}}f$,
	respectively. Here  we have used the Einstein convention of summation over repeated indices, and $\partial_k:=\partial_{y_k}$.
	
	Setting $J=\det \nabla \zeta$, the first expression in \eqref{1.7} implies that
	$\det F |_{x=\zeta}=J$. By the direct calculation, we have
	$$\partial_k(J\mathcal{A}_{ik})=0,$$
	which, together with the relation $\mathcal{A}^{\top}
	\nabla\zeta =I$, leads to
	\begin{align}
		\mm{div}( {FF^{\top}}/{\det F}) |_{x=\zeta}=&
		\mm{div}_{\mathcal{A}}( J^{-1}\nabla \zeta \nabla
		\zeta^{\top} )=J^{-1}
		\mm{div}({\mathcal{A}}^{\top}(\nabla \zeta \nabla
		\zeta^{\top})) =J^{-1}\Delta\zeta .  \label{1.9}
	\end{align}
	Under the condition \eqref{1.6}, thus the Cauchy problem of the system \eqref{1.1} in Lagrangian coordinates can be rewritten as follows:
	\begin{equation}
		\label{1.17}
		\begin{cases}
			\zeta_t =u,\\
			\varrho_t+\varrho\mm{div}_{\ml{A}} u=0 ,\\
			\varrho u_t+ \nabla_{\ml{A}} P(\varrho)= \mu
			\Delta_{\ml{A}} u
			+\lambda \nabla_{\ml{A}} \mm{div}_{\ml{A}}u +\kappa J^{-1} \Delta
			\zeta  ,\\
			(\zeta, \varrho, u)|_{t=0} = (\zeta^0, \varrho^0, u^0),
		\end{cases}
	\end{equation}
	where $(\varrho, u)(y,t) := (\rho, v) (\zeta(y, t), t)$.
	Obviously $(\zeta,{\varrho},u)=(y,\bar{\rho},0)$ represents the rest state of the equations \eqref{1.17}$_1$--\eqref{1.17}$_3$, where $\bar{\rho}$ is a positive constant.
	
	Viscoelasticity is a material property that exhibits both viscous and elastic characteristics when undergoing deformation.
	In particular, a viscoelastic fluid strains when stretched, and quickly returns to its rest state
	once the stress is removed. Therefore, we naturally have the following asymptotic behaviors after the rest state of the viscoelastic fluid being perturbed:
	\begin{equation}
		\label{1.18}
		\zeta(y,t)\to y\mbox{ and } \varrho(y,t)\to \bar{\rho}
		\mbox{ as }t\to \infty.
	\end{equation}
	It holds from $\eqref{1.17}_1$ that
	\begin{equation}
		\label{1.19}
		J_t=J\mm{div}_{\mathcal{A}}u,
	\end{equation}
	which, combining with $\eqref{1.17}_2$, implies that
	\begin{equation}
		\label{1.20}
		\partial_t(\varrho J)=0.
	\end{equation}
	Thanks to \eqref{1.18} and
	\eqref{1.20}, we can obtain that
	$$\varrho^0J^0=\varrho
	J=\bar{\rho},$$ which yields $\varrho = \bar{\rho}J^{-1}$
	for the initial data $(\varrho^0,J^0)$ satisfying
	\begin{align}
		\label{abjlj0i}
		\varrho^0 =\bar{\rho}J^{-1}_0,
	\end{align}
	which is called the asymptotic stability condition of $\varrho^0$.
	\emph{Without loss of generality, in what follows we take $ \bar{\rho} =1$ for the sake of the simplicity}.
	Thanks to the additional asymptotic stability condition of $\varrho^0$ in \eqref{abjlj0i}, the
	Cauchy problem \eqref{1.17} reduces to
	\begin{equation}
		\label{1.17n}
		\begin{cases}
			\zeta_t =u,\\
			J^{-1} u_t+ \nabla_{\ml{A}} P(\varrho)= \mu
			\Delta_{\ml{A}} u
			+\lambda \nabla_{\ml{A}} \mm{div}_{\ml{A}}u +\kappa J^{-1}\Delta
			\zeta ,\\
			(\zeta, \varrho, u)|_{t=0} = (\zeta^0, \varrho^0, u^0),
		\end{cases}
	\end{equation}
	Let $\eta(y, t):= \zeta(y, t) -y$ and $\mathcal{A}:=(\nabla\eta+I)^{-\top}$. The problem \eqref{1.17n} can be further rewritten into the following
	nonhomogeneous form:
	\begin{equation}
		\label{1.21}
		\begin{cases}
			\eta_t =u,\\
			u_t -\nabla( P'(1)
			\mm{div}\eta+\lambda\mm{div} u)-\Delta(\mu u +\kappa\eta)=
			\mathcal{N},\\
			(\eta, u)|_{t =0} = (\eta^0, u^0),
		\end{cases}
	\end{equation}
	where $\tilde{\mathcal{A}}
	:=\mathcal{A} -I$ and
	\begin{align}
		\mathcal{N}:= &\mu  (
		\mm{div}_{\tilde{\ml{A}}}\nabla_{\ml{A}}u
		+ \mm{div} \nabla_{\tilde{\ml{A}}}u)+\lambda  (
		\nabla_{\tilde{\ml{A}}}\mm{div}_{\ml{A}}u
		+ \nabla\mm{div}_{\tilde{\ml{A}}}u) \nonumber\\
		& + (J-1)(\mu \Delta_{\mathcal{A}}u+\lambda
		\nabla_{\mathcal{A}}\mm{div}_{\mathcal{A}}u)-(J-1)
		\nabla_{\mathcal{A}}
		P( J^{-1})-\nabla_{\tilde{\mathcal{A}}}
		P( J^{-1})  \nonumber\\
		& -\nabla \left(
		{P}'(1) (J^{-1}-1+\mm{div}\eta)
		+\int_{0}^{ (J^{-1}-1)}(
		(J^{-1}-1)-z)\frac{\mm{d}^2}{\mm{d}z^2} P
		(1 +z)\mm{d}z\right).
		\label{1.4}
	\end{align}
	
	\subsection{Main results}
	
	Before stating our main results, we shall introduce
	basic notations which will be used repeatedly throughout this paper.
	\begin{enumerate}[(1)]
		\item Simplified notations:
		\begin{align*}
			& I_t := (0, t), \ \overline{I_t}:=\mbox{the closure of } I_t, \
			\mathbb{R}_0^+ := [0, \infty),
			\nonumber \\& L^p  = W^{0, p}(\mathbb{R}^3),\  H^m := W^{m, 2}(\mathbb{R}^3),\ \|\cdot\|_m:=\|\cdot\|_{H^m(\mathbb{R}^3)},\\
			& (f \left|\, g\right.)_{H^m}
			: =\sum\nolimits_{|\alpha| \leqslant m}
			\int_{\mathbb{R}^3}\partial^\alpha f \cdot
			\partial^\alpha g \mm{d}y,\ f\lesssim g\mbox{ represents that }f\leqslant cg,\\
			&f\approx
			g \mbox{ means  that }f\lesssim g\mbox{ and }g\lesssim f\mbox{ hold}.
		\end{align*}
		Here $ t\in \mathbb{R}^+$, $1\leqslant p\leqslant \infty$,
		and  $m$ is a non-negative integer. The letter $c>0$ denotes a general constant that depends at most on $P(\cdot)$, $\mu$ and $\lambda$, and it may vary from line/place to line/place. However sometimes we also renew to denote the constant $c$ by $c_i$ for $1\leqslant i\leqslant 4$, under such case, $c_i$ does not vary from line/place to line/place. 
		In addition,
		\begin{align}
			& H_*^j: = \{\xi \in H^j~|~\phi:= \xi(y, t) +y: \mathbb{R}^3
			\rightarrow \mathbb{R}^3 \text{ is a } C^1(\mathbb{R}^3)
			\text{-diffeomorphic mapping}\nonumber\\
			&\qquad \qquad \qquad\quad \mbox{ and }\det(\nabla\phi) \geqslant 1/2\},\nonumber
		\end{align}
		where $j$ is a positive integer.
		\item Energy/dissipation functionals:
		\begin{align}
			\mathcal{E}(t) &: = \|u\|_{2}^2
			+\kappa \|\nabla \eta\|_2^2 +(t +1) (\|\nabla
			u\|_2^2 + \kappa \|\Delta \eta\|_2^2)\nonumber\\
			&\quad  +(t +1)^2 (\|\nabla^2 u\|_1^2 +\kappa
			\|\Delta \nabla  \eta\|_1^2), \nonumber\\
			\mathcal{D}(t) & := \|\nabla
			u\|_2^2 +\kappa \|\Delta \eta\|_2^2 +(t
			+1)(\|\Delta u\|_2^2  +\kappa \|\nabla \Delta
			\eta\|_1^2) +(t
			+1)^2 \|\nabla \Delta u\|_1^2.\nonumber
		\end{align}
	\end{enumerate}
	
	Now we state our first  result for the global existence of  unique classical
	solutions of the Cauchy problem \eqref{1.21}, where the initial velocity is smaller than  the elasticity coefficient.
	\begin{thm}\label{1.12}
		Let $\eta^0 \in H^4 \cap H^4_*$ and $u^0 \in
		H^3$. There are two positive constant $c_1$ (sufficiently small)
		and $c_2$ such that, if the elasticity coefficient $\kappa$ is relatively large and satisfies the following condition:
		\begin{align}
			\kappa^{-1} \max\left\{\sqrt{2 c_2
				\mathcal{E}^0}, (2 c_2\mathcal{E}^0)^2 \right\} \leqslant {c_1},\label{1.22}
		\end{align}
		then the Cauchy problem \eqref{1.21} admits a unique
		global classical solution $(\eta, u) \in C^0(\mathbb{R}_0^+; H^4 \times H^3)$ satisfying that,
		for any $t >0$,
		$\eta(t)\in H^4_*$ and
		\begin{align}
			\mathcal{E}(t) + \int_0^t \mathcal{D}(\tau) \mm{d} \tau \lesssim  \mathcal{E}^0 .\label{0.11}
		\end{align}
	\end{thm}
	
	We briefly mention the proof of Theorem \ref{1.12} to highlight its key steps. It is well-known that the linearized problem of \eqref{1.21} enjoys the following basic energy identity
	\begin{align}
		{\|u\|_0^2 }  +{\kappa}\|\nabla
		\eta\|_0^2 + {P'(1) }
		\|\mm{div} \eta\|_0^2  + {2} \int_0^t
		\big(\lambda\|\div u\|_0^2 +\mu\|\nabla
		u\|_0^2\big) \mm{d} \tau = {I_0} ,
		\label{3.1}
	\end{align}
	where we have defined that $I_0:=\|u^0\|_0^2 +\kappa \|\nabla
	\eta^0\|_0^2 +P'(1) \|\mm{div}
	\eta^0\|_0^2$.  In particular, we find  that $\|\nabla \eta\|_0\to 0$ as $\kappa\to \infty$ for fixed $I^0$. Based on this key observation, Jiang--Jiang used a standard energy method (without temporal weights) to establish
	the existence result of unique strong solutions of in a spatial periodic domain (denoted by $\mathbb{T}^3$)
	whenever the elasticity coefficient is  larger than initial velocity perturbation of the rest state. Follow the spirit of Jiang--Jiang's proof, we further exploit an energy method with temporal weights to establish Theorem \ref{1.12}. The advantage of our method provides the temporal decay rates of solutions.
	These decay rates will be useful to deal with the nonlinear terms in the process of the spectral analysis for further improving the known decay rate of spatial derivatives of solutions in \eqref{0.11}.
	
	We mention that the existence result in Lagrangian coordinates from Theorem \ref{1.12} can be recovered
	to the one in Eulerian coordinates. In fact, noting that the solution $\eta$ in Theorem \ref{1.12} satisfies
	\begin{align}
		\zeta:= \eta(y, t)+y   : \mathbb{R}^3 \to \mathbb{R}^3 \mbox{ is a }C^2\mbox{ diffeomorphism mapping and }\det(\nabla\zeta) \geqslant 1/2, \label{2312018031adsadfa21601xx}
	\end{align}
	and then using an inverse transform of the Lagrangian coordinates, i.e., $$(\rho,v,F):=( J^{-1},u,\nabla \zeta)|_{y=\zeta^{-1}},$$ we can easily
	get a global solution of \eqref{1.1} from Theorem \ref{1.12}. More precisely, we have the following conclusion,  the proof of which
	is referred to \cite[Theorem 1.2]{JFJSJMFMOSERT}.
	\begin{cor}\label{201904301948xx}
		Let  $ (\rho^0, v^0,F^0)\in H^3  $  and $\kappa$ satisfy the following conditions:
		\begin{enumerate}[(1)]
			\item  $F^{0}=\nabla\zeta^{0}(\zeta_{0}^{-1}(x))$, where $\eta^{0}:=\zeta^{0}(y)-y\in H^4_{*}$,
			\item \eqref{1.22} holds with $v^0(\zeta^0)$ in place of $u^0$.
		\end{enumerate}
		Then, the Cauchy problem of \eqref{1.1} with initial data $(\rho,v,F)|_{t=0}=(\rho^0,v^0,F^0)$ admits a unique global classical solution $(\rho,v,F)\in C^0(\mathbb{R}_0^+,H^3)$.
	\end{cor}
	
	As mentioned above, the solution $(\eta,u)$ to \eqref{1.12} can be approximately by
	the solution of the corresponding linear problem of \eqref{1.21} for  sufficiently large $\kappa$ and $\mathcal{E}^0$ being uniformly bounded with respect to $\kappa$. Such property is also called the   vanishing phenomena of the nonlinear interactions. More precisely, we have the following conclusion for the
	asymptotic behaviour of global solutions in Theorem \ref{1.12} with respect to $\kappa$.
	\begin{thm}
		\label{1.10}
		Let $(\eta, u)$ be the solution given by Theorem \ref{1.12}.
		Then the following linear problem
		\begin{equation}
			\label{1.13}
			\begin{cases}
				\eta^l_t =u^l,\\
				u^l_t -\nabla( P'(1)
				\mm{div}\eta^l +\lambda\mm{div} u^l)-\Delta(\mu u^l
				+\kappa\eta^l)= 0,\\
				(\eta^l, u^l)|_{t =0} = (\eta, u)|_{t=0}
			\end{cases}
		\end{equation}
		admits a unique solution   $(\eta^l,u^l)\in C^0(\mathbb{R}_0^+; H^4 \times H^3)$. Moreover,  for $(\eta^d, u^d) =(\eta -\eta^l, u -u^l)$ and any given $t>0$, it holds
		\begin{align}
			\bar{\mathcal{E}}(t) + \int_0^t \bar{\mathcal{D}}(\tau) \mm{d} \tau \lesssim \kappa^{-\frac{1}{2}}
			\big(\mathcal{E}_0^\frac{3}{2} +
			\mathcal{E}_0^3 \big),
			\label{1.14}
		\end{align}
		where $\bar{\mathcal{E}}(t)$ and $\bar{\mathcal{D}}(t)$ are defined by $ {\mathcal{E}}(t)$ and $ {\mathcal{D}}(t)$ with $(\eta^d, u^d)$ in place of  $(\eta, u)$, respectively.
	\end{thm}
	
	Theorem \ref{1.10} can be easily obtained by following the proof of Theorem \ref{1.12}, we will briefly sketch the proof in Section \ref{202240110261721}. Finally, exploiting the spectral analysis
	for solutions obtained in Theorem \ref{1.12} and the careful energy estimates for the nonlinear terms in $\mathcal{N}$,  we further improve the  temporal decay rates of solutions as follows:
	\begin{thm}\label{1.24}
		Let  $(\eta, u)$ be the solution given by Theorem \ref{1.12} and $(\eta^0, u^0)$ satisfy \eqref{1.22}.  Additionally suppose that the norm $\|(\eta^0, u^0)\|_{L^1}$ is bounded. For sufficiently small $c_1$ in  \eqref{1.22}, it holds
		\begin{align}
			& \|\nabla^k (\eta(t), u(t)) \|_0  \lesssim \left(\|(\eta^0, u^0)\|_{L^1} +
			{\sqrt{\mathcal{E}^0}}\right)(1 +t)^{-\frac{3}{4} - \frac{k}{2}}, \label{1.26}
		\end{align}
		where $k =0$, $1$ and $2$.
	\end{thm}
	
	Similarly to Corollary \ref{201904301948xx}, we can easily deduce the corresponding version in Eulerian coordinates from  the above theorem.
	\begin{cor}
		Let $(\rho,v,F)$ be the solution provided by Corollary \ref{201904301948xx} with sufficiently small $c_1$, then $(\rho,v,F)$ enjoys the following estimates
		\begin{align}
			& \|\nabla^k (\rho(t) - 1, F(t)-I)\|_0 \lesssim (\|(\eta^0, u^0)\|_{L^1} +
			\sqrt{\mathcal{E}^0}) (1 +t)^{-\frac{3}{4} -\frac{k+1}{2}}, \label{1.27x}\\
			& \|\nabla^l v(t)\|_0 \lesssim (\|(\eta^0, u^0)\|_{L^1} +
			\sqrt{\mathcal{E}^0}) (1 +t)^{-\frac{3}{4} -\frac{l}{2}}, \label{1.28}
		\end{align}
		where $k=0$, $1$ and $l=0$, $1$, $2$.
	\end{cor} 
	\begin{rem} Similarly the temporal decay of $\|\nabla^k (\rho - 1, F-I)\|_0 $ with $k=2$, $3$ and $\|\nabla^3 v\|_0$ can be easily computed out by using \eqref{0.11}.
			However, interesting readers can further refer to \cite{MR4641382} for the faster temporal decay rates of the highest-order spacial derivatives of solutions of the system \eqref{1.1}.
			In addition, following the similar arguments introduced by Ishigaki in \cite{MR4152207}, we can further derive that  
			the    solution in Corollary \ref{201904301948xx} with sufficiently small $c_1$ further satisfies
			\begin{align}
				\|(\rho(t) -1, F(t)-I)\|_{L^\infty} \lesssim (\|(\eta^0, u^0)\|_{L^1} +
				\sqrt{\mathcal{E}^0})(1 +t)^{-2} ;\label{1.23}
			\end{align}
			moreover, in view of \eqref{1.27x} with $k =0$ and \eqref{1.23},  it is easy to see  that for $p >2$,
		\begin{align} 
			\|(\rho(t) -1, F(t)-I)\|_{L^p} & \lesssim \|(\rho(t) -1, F(t)-I)\|_0^\frac{2}{p} \|(\rho(t) -1, F(t)-I)\|_{L^\infty}^{1 - \frac{2}{p}} \nonumber\\
			& \lesssim (\|(\eta^0, u^0)\|_{L^1} +
			\sqrt{\mathcal{E}^0}) (1 +t)^{-\frac{3}{2}\left(1 -\frac{1}{p}\right)-\frac{1}{2}}. 
		\end{align}
	\end{rem}
	
	The rest three sections of this paper are devoted to the proofs of
	Theorems \ref{1.12}--\ref{1.24} respectively.
	
	\section{Proof of Theorem \ref{1.12}}
	
	In this section, we will prove Theorem \ref{1.12} by three steps. First, we   \emph{a priori} derive the stability estimate \eqref{0.11} in Section \ref{2}. Then, we introduce the local(-in-time) solvability of the Cauchy problem \eqref{1.21}  (see Proposition \ref{2.23}) and the theorem of homeomorphism mappings (see Proposition \eqref{2.24}) in Section \ref{3}.
	Finally,  we establish the global(-in-time)  solvability of  \eqref{1.21} stated in Theorem \ref{1.10} by a continuity method in Section \ref{3x}.
	
	\subsection{$\textit{A priori}$ stability  estimates}\label{2}
	
	This section focuses on deriving the \emph{a priori} stability estimate
	\eqref{0.11} for the classical solution $(\eta,u)$ defined on $\mathbb{R}^3\times \overline{I_T}$ to the Cauchy problem
	\eqref{1.21}, where $T$ is a given positive constant.  To begin with, we derive the (generalized) energy inequality for  $(\eta,u)$.
	\begin{lem}
		\label{3.12}
		It holds for any $t \in \overline{I_T}$ that
		\begin{align}
			\mathcal{E}(t) + \int_0^t \mathcal{D}(\tau) \mm{d}\tau & \lesssim \mathcal{E}^0(\eta^0, u^0)
			+\int_0^t \mathcal{I}_1(\tau) \mathrm{d}\tau, \label{3.11}
		\end{align}
		where $\alpha: = \min\left\{ {\mu
			P'(1)}/{8\lambda}, {\kappa}/{8 },
		{\mu^2}/{32}\right\}$ and
		\begin{align}
			\mathcal{I}_1(t):=& \left( \mathcal{N}  \left|\,   u
			-\frac{\mu}{4}\Delta \eta
			-\frac{\mu}{8}(\tau +1)\Delta u  \right.\right )_{H^2} + \left( \mathcal{N}  \left|\, \alpha (\tau +1)
			\Delta^2 \eta +\frac{\alpha}{2}(\tau +1)^2\Delta^2 u
			\right.\right )_{H^1} .
		\end{align}
		
	\end{lem}
	\begin{proof}
		Taking the inner product of the equation
		$\eqref{1.21}_2$ and $u$ in $H^2$, and then integrating by
		parts, we have
		\begin{align}
			\frac{1}{2}\frac{\mm{d}}{\mm{d}
				t}\big( \|u\|_2^2 +\kappa\|\nabla
			\eta\|_2^2 +P'(1)   \|\mm{div}
			\eta\|_2^2 \big) +\lambda\|\div u\|_2^2
			+\mu\|\nabla u\|_2^2 = (\mathcal{N} \left|\, u \right.)_{H^2}.
			\label{3.2}
		\end{align}
		Similarly, taking the inner product of $\eqref{1.21}_2$ and $-\Delta \eta$
		in $H^2$ leads to
		\begin{align}
			\label{2.1}
			& \frac{1}{2}\frac{\mm{d}}{\mm{d} t}\big( \mu \|\Delta
			\eta\|_2^2 + \lambda \|\nabla \mm{div} \eta\|_2^2
			\big) +P'(1)  \|\nabla\div \eta\|_2^2
			\nonumber\\
			&  + \kappa \|\Delta \eta\|_2^2 - (u_t \left|\,
			\Delta \eta\right.)_{H^2} = - (\mathcal{N} \left|\,\Delta \eta\right.)_{H^2}.
		\end{align}
		Since
		$$(u_t \left|\,\Delta \eta\right.)_{H^2} =\frac{\mm{d}}{\mm{d} t}(u \left|\,
		\Delta\eta\right.)_{H^2} + \|\nabla u\|_2^2,$$
		it holds from \eqref{2.1} that
		\begin{align}
			& \frac{1}{2}\frac{\mm{d}}{\mm{d} t}\big(\mu \|\Delta
			\eta\|_2^2 + \lambda \|\nabla \mm{div} \eta\|_2^2
			-2   (u\left|\,\Delta
			\eta\right. )_{H^2}\big) +P'(1)  \|\nabla\div
			\eta\|_2^2 \nonumber\\
			&   + \kappa \|\Delta \eta\|_2^2
			- \|\nabla u\|_2^2= -(\mathcal{N} \left|\,\Delta
			\eta\right. )_{H^2}.\label{2.2}
		\end{align}
		The sum of \eqref{3.2} and $
		\eqref{2.2}\times{\mu}/{4  } $ implies that
		\begin{align}
			& \frac{\mm{d}}{\mm{d} t} \mathscr{E}_1(t)
			+\frac{\mu}{4}\left(3\|\nabla u\|_2^2+ {  P'(1) } \|\nabla\div
			\eta\|_2^2+ { \kappa} \|\Delta
			\eta\|_2^2\right) \nonumber\\
			&  +\lambda\|\mm{div}
			u\|_2^2 =\left(\mathcal{N} \left|\, u
			-\frac{\mu}{4 }\Delta \eta \right.\right)_{H^2}, \label{3.7}
		\end{align}
		where we have defined that
		\begin{align}
			\mathscr{E}_1(t) &: =
			\frac{1}{2} \left( \|u\|_2^2 + \kappa \|\nabla
			\eta\|_2^2 + P'(1)
			\|\mm{div} \eta\|_2^2 +\frac{\mu}{4}\left(\mu \|\Delta \eta\|_2^2 \right.\right.\nonumber\\
			&\quad \left.+  \lambda \|\nabla \mm{div}
			\eta\|_2^2 -2 (u \left|\,\Delta \eta\right. )_{H^2}\right)\Big). \nonumber
		\end{align}
		
		Taking the inner product of the equation $\eqref{1.21}_2$ and $-(t +1)\Delta
		u$ in $H^2$, one has
		\begin{align}
			&\frac{\mm{d}}{\mm{d} t} \mathscr{E}_2(t)
			-\frac{1}{2} \left( \|\nabla u\|_2^2 +P'(1)\|\nabla \div \eta\|_2^2 +\kappa \|\Delta \eta\|_2^2 \right)
			\nonumber\\
			& + (t +1)(\lambda\|\nabla \div
			u\|_2^2 +\mu \|\Delta
			u\|_2^2 )= -(t +1) (\mathcal{N}
			\left|\,\Delta u\right. )_{H^2}, \label{3.4}
		\end{align}
		where
		$$
		\begin{aligned}
			\mathscr{E}_2(t): = \frac{ (t +1) }{2}\left(\|\nabla
			u\|_2^2 + {  P'(1)}  \|\nabla
			\mm{div} \eta\|_2^2 + {\kappa}  \|\Delta \eta\|_2^2\right).
		\end{aligned}
		$$
		
		Analogously, taking the inner product of
		$\eqref{1.21}_2$ and $ (t +1)\Delta^2 \eta$ in $H^1$ arises that
		\begin{align}
			& \frac{\mm{d}}{\mm{d} t} \mathscr{E}_3(t)
			+(t +1)(\kappa\|\Delta\nabla \eta \|_1^2
			+P'(1)    \|\nabla^2 \div
			\eta\|_1^2-\|\nabla^2 u\|_1^2 )\nonumber\\
			&-\frac{1}{2} \left( \mu\|\Delta\nabla \eta\|_1^2 + \lambda \|\nabla^2 \div \eta\|_1^2 \right)
			= (t +1)(\mathcal{N} \left|\, \Delta^2
			\eta \right.)_{H^1},
			\label{3.5}
		\end{align}
		where we have defined that
		$$
		\begin{aligned}
			\mathscr{E}_3(t) : =& \frac{1}{2}\left((t +1)(\mu
			\|\Delta\nabla \eta\|_1^2 + {\lambda}  \|\nabla^2 \div
			\eta\|_1^2 -
			(\nabla\Delta\eta\left|\, \nabla u\right.)_{H^1} )- {\|\nabla^2
				\eta\|_1^2 }\right) .
		\end{aligned}
		$$
		
		At the last step, by taking the
		inner product of the equation $\eqref{1.21}_2$ and $(t
		+1)^2\Delta^2 u$ in $H^1$, we get that
		\begin{align}
			&\frac{\mm{d}}{\mm{d} t} \mathscr{E}_4(t) -
			(t +1)(\|\nabla^2 u\|_1^2 +
			\kappa \|\nabla \Delta \eta\|_1^2
			+P'(1)  \|\nabla^2
			\div \eta\|_1^2) \nonumber\\
			& +(t +1)^2 (\mu\|\nabla \Delta
			u\|_1^2 + \lambda \|\nabla^2
			\div u\|_1^2 )= (t +1)^2 (\mathcal{N} \left|\,
			\Delta^2 u\right.)_{H^1}. \label{3.6}
		\end{align}
		where we have defined that
		$$
		\begin{aligned}
			\mathscr{E}_4(t): = \frac{(t
				+1)^2}{2}\left(  \|\nabla^2 u\|_1^2 + {\kappa}  \|\Delta \nabla  \eta \|_1^2 + {
				P'(1) } \|\nabla^2 \div
			\eta\|_1^2\right).
		\end{aligned}
		$$
		Calculating $\eqref{3.7} +\mu  8^{-1} \eqref{3.4} +\alpha  \eqref{3.5} + \alpha 2^{-1} \eqref{3.6}$ with $\alpha = \min\left\{{\mu
			P'(1)}/{ \lambda}, {\kappa},
		{\mu^2}/{4}\right\}/8$, and then canceling the
		dissipative terms with negative sign, we obtain that
		\begin{align}
			\frac{\mm{d}}{\mm{d} t} \mathscr{E}(t)
			+\mathscr{D}(t) & \leqslant \mathcal{I}_1(t),
			\label{3.8}
		\end{align}
		where
		\begin{align}
			\mathscr{E}(t) & :=\mathscr{E}_1(t) +{\mu}/{8 }\mathscr{E}_2(t) +\alpha \mathscr{E}_3(t) +{\alpha}/{2}\mathscr{E}_4(t)\nonumber\\
			& \geqslant \frac{1 }{2}\left(\|u\|_2^2
			+\kappa\|\nabla \eta\|_2^2
			+{P'(1) } \|\mm{div}
			\eta\|_2^2 +\frac{7\mu^2}{32} \|\Delta
			\eta\|_2^2+ \frac{\mu\lambda}{4  } \|\nabla \mm{div}
			\eta\|_2^2\right)\nonumber\\
			&\quad -\frac{\mu}{4} (u \left|\right. \Delta \eta)_{H^2} +\frac{t+1}{2}\bigg(\frac{1}{8} \left(\mu\|\nabla u\|_2^2
			+ {\mu P'(1)}   \|\nabla \mm{div}
			\eta\|_2^2+  {\mu\kappa} \|\Delta
			\eta\|_2^2\right)  \nonumber\\
			&\quad  +{\alpha\mu} \|\Delta\nabla
			\eta \|_1^2 +{\alpha\lambda} \|\nabla^2
			\div  \eta\|_1^2
			-2\alpha  (\nabla\Delta\eta \left|\right. \nabla
			u)_{H^1} \bigg)\nonumber\\
			& \quad +\frac{(t +1)^2}{{4} } \left({\alpha }\|\nabla^2
			u\|_1^2 +{\alpha\kappa}  \|\Delta \nabla
			\eta\|_1^2 +{\alpha P'(1) } \|\nabla^2 \div \eta\|_1^2 \right)
		\end{align}
		and
		\begin{align}
			\mathscr{D}(t):= & \frac{\mu}{2}\|\nabla
			u\|_2^2 +\lambda \|\div u\|_2^2 + \frac{\mu
				P'(1)}{8} \|\nabla \div \eta\|_2^2 +\frac{\mu
				\kappa}{8 } \|\Delta \eta\|_2^2 \nonumber\\
			& +\frac{t +1}{2 } \left(\frac{\mu\lambda}{4 }\|\nabla \div
			u\|_2^2 +\frac{\mu^2}{8 } \|\Delta
			u\|_2^2 +{\alpha\kappa}\|\Delta \nabla
			\eta\|_1^2 + {\alpha P'(1)  } \|\nabla^2 \div \eta \|_1^2\right)  \nonumber\\
			&+\frac{\alpha(t
				+1)^2}{2}\left(\mu\|\nabla \Delta u\|_1^2  + { \lambda} \|\nabla^2 \div u\|_1^2\right) \approx  \mathcal{D}(t).
			\label{2.6}
		\end{align}
		By H\"older's and Young's inequalities, we
		deduce
		$${\mu} \big(u \left|\,\Delta\eta\right. \big)_{H^2} \leqslant
		\frac{
			4\|u\|_2^2 }{3}+\frac{3\mu^2}{16 }\|\Delta
		\eta\|_2^2 $$
		and
		$$\alpha   (\nabla\Delta\eta\left|\, \nabla
		u\right.)_{H^1} \leqslant \frac{\alpha \mu}{4} \|\nabla\Delta\eta\|_1^2 +\frac{\mu}{32} \|\nabla u\|_1^2.$$
		Consequently, we further obtain the following lower bound for
		$\mathscr{E}(t)$.
		\begin{align}
			\mathscr{E}(t) & \geqslant
			\frac{ \|u\|_2^2}{6} +\frac{\kappa}{2}\|\nabla
			\eta\|_2^2 +\frac{P'(1) }{2}
			\|\mm{div} \eta\|_2^2 +\frac{\mu^2}{16  }
			\|\Delta \eta\|_2^2 \nonumber\\
			&\quad + \frac{\mu\lambda}{8  } \|\nabla \mm{div}
			\eta\|_2^2 +\frac{t +1}{2} \left( {\mu}\left( P'(1)  \|\nabla \mm{div}
			\eta\|_2^2
			+ { \|\nabla u\|_2^2}/{2}\right. \right.\nonumber\\
			&\quad \left.\left.+  \kappa \|\Delta
			\eta\|_2^2 +4{\alpha }  \|\Delta\nabla
			\eta\|_1^2\right) + {\alpha\lambda}  \|\nabla^2 \div
			\eta\|_1^2\right)/8\nonumber\\
			& \quad +\frac{\alpha (t +1)^2 }{4}\left(\|\nabla^2
			u\|_1^2 + \kappa  \|\Delta \nabla
			\eta\|_1^2  +   P'(1)  \|\nabla^2 \div \eta\|_1^2\right) \nonumber\\
			& \geqslant \frac{\|u\|_2^2 }{6}
			+\frac{\kappa}{2}\|\nabla \eta\|_2^2
			+\frac{P'(1) }{2} \|\mm{div}
			\eta\|_2^2 +\frac{\mu^2}{16  } \|\Delta
			\eta\|_2^2 \nonumber\\
			&\quad + \frac{\mu\lambda}{8  } \|\nabla \mm{div}
			\eta\|_2^2 +\frac{\mu(t +1) }{16}\left( \frac{\|\nabla u\|_2^2}{2}
			+ {\kappa} \|\Delta
			\eta\|_2^2\right) \nonumber\\
			&\quad +\frac{\alpha (t +1)^2 }{4}\left(\|\nabla^2
			u\|_1^2 + { \kappa}  \|\Delta \nabla
			\eta\|_1^2 \right)\gtrsim \mathcal{E}(t).
			\label{3.9}
		\end{align}
		Besides, it is easy to deduce that
		\begin{align}
			\mathscr{E}^0 \lesssim \|(u^0,\nabla
			\eta^0,\mm{div} \eta^0)\|_{3}^2 +\|\Delta
			\eta^0\|_2^2 \lesssim \mathcal{E}^0.
			\label{3.10}
		\end{align}
		Integrating the inequality \eqref{3.8} over $[0, t]$, and then making use of \eqref{2.6}--\eqref{3.10}, we arrive at the desired estimate \eqref{3.11}.
	\end{proof}
	
	Next we shall deal with the term $\mathcal{I}_1$ on the right hand side of the  energy inequality
	\eqref{3.11}. Obviously, $\mathcal{I}_1$ can be rewritten as follows:
	\begin{align}
		\mathcal{I}_1(t) =  \mathcal{I}_2(t)  + \mathcal{I}_3(t) 
		\label{2.8},
	\end{align}
	where we have defined that
	\begin{align*}
		\mathcal{I}_2(t):=& \left(\mathcal{N}_P\left|\,u -\frac{\mu}{4 }\Delta \eta
		-\frac{\mu}{8 }(t +1)\Delta u\right.\right)_{H^2} + \left(\mathcal{N}_P\left|\, \alpha (t +1) \Delta^2 \eta
		+\frac{\alpha}{2}(t +1)^2\Delta^2 u \right.\right)_{H^1}\nonumber \\
		\mathcal{I}_3(t) :=&\left( \mathcal{N}^u \left|\,u -\frac{\mu}{4 }\Delta \eta
		-\frac{\mu}{8 }(t +1)\Delta u \right.\right)_{H^2} + \left(\mathcal{N}^u   ~\bigg|~ \alpha (t +1) \Delta^2 \eta
		+\frac{\alpha}{2}(t +1)^2\Delta^2 u \right)_{H^1}
		\\
		\mathcal{N}_P := &-(J-1) \nabla_{\mathcal{A}}
		P( J^{-1})-\nabla_{\tilde{\mathcal{A}}}
		P( J^{-1})  \\
		& \quad -\nabla \left(
		{P}'(1) (J^{-1}-1+\mm{div}\eta)
		+\int_{0}^{ (J^{-1}-1)}(
		(J^{-1}-1)-z)\frac{\mm{d}^2}{\mm{d}z^2} P
		(1 +z)\mm{d}z\right)
	\end{align*}
	and
	\begin{align*}
		\mathcal{N}^u  :=& \mu \big(\mm{div}_{\tilde{\ml{A}}}\nabla_{\ml{A}}u
		+ \mm{div} \nabla_{\tilde{\ml{A}}}u \big) +\lambda  (
		\nabla_{\tilde{\ml{A}}}\mm{div}_{\ml{A}}u
		+ \nabla\mm{div}_{\tilde{\ml{A}}}u)\\
		&  + (J-1)(\mu \Delta_{\mathcal{A}}u+\lambda
		\nabla_{\mathcal{A}}\mm{div}_{\mathcal{A}}u).
	\end{align*}
	
	Now we further \emph{a priori} assume that
	\begin{align}
		\label{2.4}
		\underset{0 \leqslant t \leqslant T}{\sup}\mathcal{E}(t) \leqslant K^2
	\end{align}
	and
	\begin{align}
		\label{2.5}
		\kappa^{-1}\max\{K, K^4\} \in (0, \delta^2)
	\end{align} for some $K> 0$  and for  sufficiently small $\delta\in (0,1)$. It should be noted that the smallness condition of $\delta$ depending only upon the parameters $\mu$ and $\lambda$ will be repeatedly used in what follows.
	In view of \eqref{2.4} and
	\eqref{2.5}, it holds that
	\begin{align}
		\underset{0 \leqslant t \leqslant T}{\sup} \|\nabla
		\eta(t)\|_3 \leqslant \delta\in (0,1).
		\label{2.7}
	\end{align}
	
	Before estimating for $I_2$  and $I_3$, we shall derive some  preliminary estimates for $J$, $\tilde{\mathcal{A}}$, ${\mathcal{A}}$ and $\mathcal{N}^u$.
	\begin{lem}
		\label{4.58}
		Under the condition \eqref{2.7},
		there exists a  sufficiently small $\delta_1$ such that for any $\delta\leqslant \delta_1$ we have
		\begin{align}
			&  {2}^{-1} \lesssim J \lesssim 1 \label{4.30} \\
			& \|J-1\|_3 +\|J^{-1} -1\|_3 \lesssim \|\nabla \eta\|_3 \label{4.31}\\
			& \|\nabla(J^{-1} -1 +\div \eta)\|_2 \lesssim \|\nabla^2 \eta\|_2 \|\nabla^3 \eta\|_1  \label{4.33} \\
			& \|\tilde{\mathcal{A}}\|_{L^\infty} \lesssim \|\tilde{\mathcal{A}}\|_2 \lesssim\|\nabla \eta\|_2 \lesssim \|\Delta \eta\|_1, \label{4.52}\\
			& \|\mathcal{A}\|_{L^\infty} \lesssim \|\tilde{\mathcal{A}}\|_{L^\infty}+1\lesssim 1 \label{4.53}.
		\end{align}
	\end{lem}
	\begin{proof}
		It can be obtained from a direct calculation that
		\begin{align}
			\label{4.54}
			J = 1 +\div \eta + r_\eta,
		\end{align}
		where $r_\eta := ((\div \eta)^2 -\text{tr}(\nabla \eta)^2) /2+ \det \nabla \eta$. Due to the smallness of $\delta$, we get by H\"{o}lder's and Sobolev inequalities that
		\begin{align}
			\label{4.55}
		\|r_\eta\|_3 \lesssim (\|\nabla \eta\|_{L^\infty} +\|\nabla^2 \eta\|_1) (\|\nabla \eta\|_0 +\|\nabla^3 \eta\|_1) \lesssim \|\nabla \eta\|_3^2.
		\end{align}
		From \eqref{4.54} and \eqref{4.55}, we obtain the estimates \eqref{4.30} and \eqref{4.31}.
		
		By the formula \eqref{4.54}, it holds, for sufficiently small $\delta$,
		$$J^{-1} -1 +\div\eta =(1 -J^{-1} -\div \eta)(\div \eta +r_\eta) +\div \eta (\div \eta +r_\eta) -r_\eta.$$
		Applying $\|\nabla \cdot\|_2$ to the above identity, we easily get \eqref{4.33}.
		
		We further derive from \eqref{4.33} that
		\begin{align}
			\label{4.56}
			\|\nabla J^{-1}\|_1 \lesssim \|\nabla \eta\|_2.
		\end{align}
		In addition, it can be easily deduced that
		\begin{align}
			\label{4.57}
			&\|(A_{ij}^*)_{3 \times 3} -I\|_2 \lesssim \|\nabla \eta\|_2^2,\\
			& \tilde{\mathcal{A}}=J^{-1}({(A^*_{ij})_{3\times 3}-I}+(1-J)I),
			\label{2024010262121}
		\end{align}
		where $A_{ij}^*$ is the algebraic complement minor of $(i,j)$-th entry of the matrix $(\partial_j \zeta_i)_{3 \times 3}$.
		The desired estimates \eqref{4.52} and \eqref{4.53} follow from \eqref{4.30} and \eqref{4.31}, \eqref{4.56}--\eqref{2024010262121} for sufficiently small $\delta$.
	\end{proof}
	\begin{lem}
		\label{4.19}
		Under the condition \eqref{2.7} with  $\delta\leqslant \delta_1$, it holds that
		\begin{align}
			\|\mathcal{N}^u \|_2 \lesssim \left( \kappa^{-\frac{1}{2}}\mathcal{E}^\frac{1}{2}(t) + \kappa^{-2}\mathcal{E}^2(t) \right) (t +1)^{-1} \mathcal{D}^\frac{1}{2}(t).
			\label{4.20}
		\end{align}
	\end{lem}
	\begin{proof}
		Exploiting \eqref{4.54}, \eqref{2024010262121}, and Young's and Sobolev's inequalities, we can derive from the definition of $\mathcal{N}^u$ that
		\begin{align}
			\|\mathcal{N}^u \|_2 \lesssim \left( 1 + \|\Delta \eta\|_2^3 \right)
			\sum_{i,j,k,l=1}^3\|\nabla (\partial_i\eta_j\partial_ku_l)\|_2. \label{2.10}
		\end{align}
		With the aid of  H\"{o}lder's and Sobolev inequalities, one has
		\begin{align}
			\|\nabla (\partial_i\eta_j\partial_ku_l)\|_0  & \lesssim \|\nabla^2 \eta\|_0 \|\nabla u\|_{L^\infty} +\|\nabla \eta\|_{L^6} \|\nabla^2 u\|_{L^3} \nonumber\\
			& \lesssim \|\nabla^2 \eta\|_0 \|\nabla^2 u\|_0^\frac{1}{2} \|\nabla^3 u\|_0^\frac{1}{2}   \lesssim \|\Delta \eta\|_0 \|\Delta u\|_1
			\label{4.9}
		\end{align}
		and
		\begin{align}
			\|\nabla^3 (\partial_i\eta_j\partial_ku_l)\|_0 & \lesssim \|\nabla^4 \eta \nabla u\|_0 +\|\nabla^3 \eta \nabla^2 u\|_0 +\|\nabla^2 \eta \nabla^3 u\|_0 +\|\nabla \eta \nabla^4 u\|_0\nonumber\\
			& \lesssim  \|\nabla^4 \eta\|_0\|\nabla u\|_{L^\infty} +\|\nabla^3 \eta\|_{L^6} \|\nabla^2 u\|_{L^3} \nonumber\\
			&\quad  +\|\nabla^2 \eta\|_{L^\infty} \|\nabla^3 u\|_0 +\|\nabla \eta\|_{L^\infty}\|\nabla^4 u\|_0 \nonumber\\
			& \lesssim \|\nabla^4 \eta\|_0 \|\nabla^2 u\|_0^\frac{1}{2} \|\nabla^3 u\|_0^\frac{1}{2} +\|\nabla^2 \eta\|_0^\frac{1}{4} \|\nabla^4 \eta\|_0^\frac{3}{4}\|\nabla^3 u\|_0\nonumber\\
			& \quad +\|\nabla^2 \eta\|_0^\frac{1}{2} \|\nabla^3 \eta\|_0^\frac{1}{2} \|\nabla^4 u\|_0  \lesssim \|\Delta \eta\|_2 \|\Delta u\|_2.
			\label{4.5}
		\end{align}
		It follows from \eqref{4.9} and \eqref{4.5} that
		\begin{align}
			\|\nabla (\partial_i\eta_j\partial_ku_l)\|_2 & \lesssim \kappa^{-\frac{1}{2}} (t +1)^{-1} \big(\sqrt{\kappa} (t +1)^\frac{1}{2} \|\Delta \eta\|_2 \big) \big((t +1)^{\frac{1}{2}} \|\Delta u\|_2\big) \nonumber\\
			& \lesssim \kappa^{-\frac{1}{2}} (t +1)^{-1} \mathcal{E}^\frac{1}{2}(t) \mathcal{D}^\frac{1}{2}(t) .
			\label{2.9}
		\end{align}
		Note that
		$$\|\Delta \eta\|_2^3 \lesssim  \kappa^{-\frac{3}{2}} (t +1)^{-\frac{3}{2}} \mathcal{E}^\frac{3}{2}(t),$$
		we thus obtain the estimate \eqref{4.20} from \eqref{2.10} and \eqref{2.9}.
	\end{proof}
	
	With  Lemmas \ref{4.58} and  \ref{4.19} in hand, we can further estimate for the two terms $I_2$ and $I_3$ as follows.
	\begin{lem}
		\label{2.20}
		Under the assumptions of \eqref{2.4} and \eqref{2.5}, it holds that
		\begin{align}
			&\mathcal{I}_2(t) \lesssim \delta(t +1)^{-1} \mathcal{D}^\frac{1}{2}(t) +\delta \mathcal{D}(t).
			\label{2.19}
		\end{align}
	\end{lem}
	\begin{proof}
		It follows from Lemma \ref{4.58} that 
		\begin{align}
			\label{4.59}
			\|\mathcal{N}_P\|_2 & \lesssim \|J-1\|_2 \|\mathcal{A}\|_{L^{\infty}} \|\nabla P( J^{-1})\|_2 +\|\tilde{\mathcal{A}}\|_2\|\nabla P( J^{-1})\|_2  \nonumber\\
			& \quad +\| \nabla (J^{-1}-1+\mm{div}\eta) \|_2 +\left\| \nabla \int_{0}^{ (J^{-1}-1)}(
			(J^{-1}-1)-z)\frac{\mm{d}^2}{\mm{d}z^2} P (1 +z)\mm{d}z \right\|_2 \nonumber \\
			& \lesssim \kappa^{-1}(t +1)^{-1} \big( \sqrt{\kappa}(t +1)^\frac{1}{2} \|\Delta \eta\|_2 \big) \big( \sqrt{\kappa}(t +1)^\frac{1}{2} \|\nabla \Delta \eta\|_1 \big) \nonumber\\
			& \lesssim \kappa^{-1}(t +1)^{-1} \mathcal{E}^\frac{1}{2}(t) \mathcal{D}^\frac{1}{2}(t).
		\end{align}
		Using Sobolev's inequality, \eqref{2.4}, \eqref{2.5} and \eqref{4.59}, we get
		\begin{align}
			& \left(\mathcal{N}_P\left|\, u -\frac{\mu}{4 }\Delta \eta -\frac{\mu}{8 }(t +1)\Delta u\right. \right)_{H^2} \nonumber
			\\
			& \lesssim \|\mathcal{N}_P\|_2 (\|u\|_2 +\|\Delta \eta\|_2 +(t +1)\|\Delta u\|_2) \nonumber\\
			& \lesssim \left(\mathcal{E}^\frac{1}{2}(t)\left(1 + \kappa^{-\frac{1}{2}}(t +1)^{-\frac{1}{2}} \right)+(t +1)^\frac{1}{2} \mathcal{D}^\frac{1}{2}(t) \right)  \kappa^{-1}(t +1)^{-1} \mathcal{E}^\frac{1}{2}(t) \mathcal{D}^\frac{1}{2}(t)\nonumber\\
			&\lesssim \kappa^{-1}\left( (t +1)^{-1} \mathcal{E}(t)\mathcal{D}(t) \right)^\frac{1}{2}\left(
			\left((t +1)^{-1} \mathcal{E}(t) \right)^\frac{1}{2} + \mathcal{D}^\frac{1}{2}(t) \right)\nonumber\\
			& \lesssim \delta (t +1)^{-1}\mathcal{D}^\frac{1}{2}(t) +\delta \mathcal{D}(t).
			\label{2.16}
		\end{align}
		Similarly,
		\begin{align}
			&\left| \left(\mathcal{N}_P\left|\, \alpha (t +1) \Delta^2 \eta \right.\right)_{H^1}
			\right| \lesssim \|\mathcal{N}_P\|_2 ((t +1) \|\Delta \eta\|_2) \lesssim \kappa^{-\frac{3}{2}} \mathcal{E}^\frac{1}{2}(t) \mathcal{D}(t) \lesssim \delta \mathcal{D}(t).\label{2.17}
		\end{align}
		Making use of   \eqref{2.4}, \eqref{2.5}, \eqref{4.59} and the integration by parts  arises that
		\begin{align}
			&\left(\mathcal{N}_P\left|\, -\frac{\alpha}{2}(t +1)^2\Delta^2 u \right.\right)_{H^1}\nonumber
			\\
			&= -\sum\limits_{j =1}\limits^{3} \left(\left(\partial_j \mathcal{N}_P\left|\,\frac{\alpha}{2}(t +1)^2 \partial_j \Delta u \right. \right)_{L^2} +\left(\partial_j \nabla \mathcal{N}_P\left|\,\frac{\alpha}{2}(t +1)^2 \partial_j\nabla \Delta u \right. \right)_{L^2} \right). \nonumber \\
			& \lesssim  (t +1)^2\|\nabla \mathcal{N}_P\|_1\|\nabla \Delta u\|_1 \lesssim \kappa^{-1} \mathcal{E}^\frac{1}{2}(t) \mathcal{D}(t) \lesssim \delta \mathcal{D}(t).\label{2.18}
		\end{align}
		The estimate \eqref{2.19} can be derived from the sum of \eqref{2.16}--\eqref{2.18}.
	\end{proof}
	
	\begin{lem}
		\label{2.11}
		Under the assumptions of \eqref{2.4} and \eqref{2.5},  it holds that
		\begin{align}
			\mathcal{I}_3(t) \lesssim \delta (t +1)^{-1} \mathcal{D}^\frac{1}{2}(t)  + \delta \mathcal{D}(t).
			\label{2.15}
		\end{align}
	\end{lem}
	\begin{proof}
		Exploiting \eqref{2.4}, \eqref{2.5}, \eqref{4.20}, Cauchy--Schwarz's inequality and the integration by part, one has
		\begin{align}
			&\left(\mathcal{N}^u\left|\, u -\frac{\mu}{4 }\Delta \eta -\frac{\mu}{8 }(t +1)\Delta u\right. \right)_{H^2} \nonumber \\
			& \lesssim \|\mathcal{N}^u \|_2 \big(\|u\|_2 +\|\Delta \eta\|_2 +(t +1)\|\Delta u\|_2\big) \nonumber\\
			& \lesssim \big( \kappa^{-\frac{1}{2}}\mathcal{E}^\frac{1}{2}(t) + \kappa^{-2}\mathcal{E}^2(t) \big) (t +1)^{-1} \mathcal{D}^\frac{1}{2}(t) \big( \mathcal{E}^\frac{1}{2}(t)(1 + \nonumber \\
			& \quad \kappa^{-\frac{1}{2}}(t +1)^{-\frac{1}{2}}) +(t +1)^\frac{1}{2} \mathcal{D}^\frac{1}{2}(t) \big) \nonumber\\
			& \lesssim \big(\kappa^{-\frac{1}{2}} \mathcal{E}(t) + \kappa^{-2} \mathcal{E}^\frac{5}{2}(t) \big)(t +1)^{-1}\mathcal{D}^\frac{1}{2}(t)\nonumber\\
			&\quad +\big( \kappa^{-\frac{1}{2}}\mathcal{E}^\frac{1}{2} (t) +\kappa^{-2}\mathcal{E}^2(t) \big)\mathcal{D}(t) \lesssim \delta (t +1)^{-1} \mathcal{D}^\frac{1}{2}(t) + \delta \mathcal{D}(t). \label{2.12}
		\end{align}
		Similarly to  \eqref{2.12}, we obtain
		\begin{align}
			&\left(\mathcal{N}^u \left|\, \alpha (t +1) \Delta^2 \eta \right.\right)_{H^1}
			\nonumber \\
			&=\left(\mathcal{N}^u \left|\, \alpha (t +1)\Delta^2 \eta \right.\right)_0 -\left(\Delta \mathcal{N}^u \left|\, \alpha (t +1)\Delta^2 \eta \right.\right)_{L^2} \nonumber \\  &\lesssim \|\mathcal{N}^u \|_2 \|(t +1) \Delta \eta\|_2 \lesssim \left( \kappa^{-1}\mathcal{E}^\frac{1}{2}(t) + \kappa^{-\frac{5}{2}}\mathcal{E}^2(t) \right) \mathcal{D}(t)  \lesssim \delta \mathcal{D}(t),
			\label{2.13}
		\end{align}
		and
		\begin{align}
			&\left( \mathcal{N}^u \left|\right. {\alpha}(t +1)^2 \Delta^2 u /2\right)_{H^1}\nonumber \\
			&= -\sum\limits_{j =1}\limits^{3} \left(\left(\partial_j \mathcal{N}^u \left|\,\frac{\alpha}{2}(t +1)^2 \partial_j \Delta u\right.\right)_{L^2} +\left(\partial_j \nabla \mathcal{N}^u \left|\, \frac{\alpha}{2}(t +1)^2 \partial_j\nabla \Delta u\right.\right)_{L^2}\right) \nonumber\\
			& \lesssim (t +1)^2 \|\mathcal{N}^u\|_2\|\nabla \Delta u\|_1 \nonumber\\
			&  \lesssim \left( \kappa^{-\frac{1}{2}}\mathcal{E}^\frac{1}{2}(t) + \kappa^{-2}\mathcal{E}^2(t) \right)  \mathcal{D}(t) \lesssim \delta \mathcal{D}(t).
			\label{2.14}
		\end{align}
		Consequently, we can obtain the desired estimate \eqref{2.15} from \eqref{2.12}--\eqref{2.14}.
	\end{proof}
	
	Thanks to Lemmas \ref{2.20} and \ref{2.11} , we can directly obtain from \eqref{2.8} that
	\begin{align}
		\int_0^t \mathcal{I}_1(\tau) \mm{d}\tau \lesssim \int_0^t \delta (\tau +1)^{-1} \mathcal{D}^\frac{1}{2}(\tau) \mm{d}\tau +\delta \int_0^t \mathcal{D}(\tau) \mm{d}\tau  \lesssim \delta +\delta \int_0^t \mathcal{D}(\tau) \mm{d}\tau.
		\label{2.21}
	\end{align}
	Exploiting \eqref{3.11} and \eqref{2.21}, there exist constants $c_2\geqslant 1$ and $\delta_2\in (0,1)$, such that, for any $\delta \leqslant \delta_2$ and $t \in \overline{I_T}$, it holds
	\begin{align}
		\mathcal{E}(t) + \int_0^t \mathcal{D}(\tau) \mm{d}\tau \leqslant c_2 \mathcal{E}^0.
		\label{2.22}
	\end{align}
	In particular, let
	\begin{align}
		K = \sqrt{2 c_2 \mathcal{E}^0},
		\label{2.25}
	\end{align}
	we can derive from \eqref{2.22} that
	\begin{align}
		\underset{0 \leqslant t \leqslant T}{\sup}\mathcal{E}(t) \leqslant K^2 / 2.
		\label{3.13}
	\end{align}
	
	\subsection{Local solvability of unique solutions and homeomorphism mappings}\label{3}

	To further establish the global well-posedness of the Cauchy problem \eqref{1.21}, we shall introduce the following two conclusions.
	\begin{pro}
		\label{2.23}
		Assume the initial data $(\eta^0, u^0)\in H^4 \times H^3$ and $\|u^0\|_3\leqslant B$ for some constant $B >0$. Then, there exists a small enough constant $\delta_3 \in (0, \delta_1/2]$, such that for any given $\|\nabla \eta^0\|_3 \leqslant \delta_3$, the Cauchy problem \eqref{1.21} admits a unique local strong solution $( \eta, u)$ in the sense that $( \eta, u)\in C^0(\overline{I_T}, H^4 )\times\mathfrak{U}_T$ with some positive constant $T$ depending only on $B$, $P(\cdot)$,   $\mu$ and $\lambda$, where
		$$   \mathfrak{U}_T := \{u \in C^0(\overline{I_T};\
		H^3(\mathbb{R}^3)) ~|~ (u,u_t)\in L^2(I_T;\
		H^{3}(\mathbb{R}^3)\times H^4(\mathbb{R}^3)) \}
		. $$Moreover, it holds $\underset{0\leqslant t \leqslant T}{\sup}\|\nabla \eta(t)\|_3 \leqslant 2\delta_3$\footnote{Here the uniqueness means that if there is another solution
			$( \tilde{\eta}, \tilde{u})\in C^0(\overline{I_T};{H}^{4} ) \times \mathfrak{U}_{T}$
			satisfying $0<\inf_{(y,t)\in \overline{\Omega\times I_T}} \det(\nabla \tilde{\eta}+I)$, then
			$(\tilde{\eta},\tilde{u})=(\eta,u)$ by virtue of the smallness condition
			$\sup_{t\in [0,T]}\|\nabla \eta\|_{3}\leqslant 2\delta_3$.}.
	\end{pro}
	\begin{pf}
		Proposition \ref{2.23} can be easily obtained by employing the same arguments as for \cite[Proposition 2.1]{MR4217912},  which has proved the local well-posedness result of the incompressible viscoelastic fluids in $\mathbb{T}^3$,  and hence its proof will be omitted here.
	\end{pf}
	\begin{pro}
		\label{2.24}
		There exists a constant $\delta_4 >0$, such that if the function $\varphi\in H^4$ satisfies $\|\nabla \varphi\|_3 \leqslant \delta_4$, then it holds (with the possibility of redefining $\varphi$ on a set of measure zero) that
		\begin{align}
			&
			\det \nabla\psi\geqslant 1/2 \mbox{ and }\psi : \mathbb{R}^3\to \mathbb{R}^3 \mbox{ is a }C^2\mbox{ homeomorphism mapping}
			,\nonumber
		\end{align}
		where $\psi:=\varphi+y$.
	\end{pro}
	\begin{pf}
		Please refer to \cite[Proposition 2.2]{MR4217912} for a proof.
	\end{pf}
	
	\subsection{Global existence of unique solutions}\label{3x}
	
	With the $\textit{a priori}$ estimate \eqref{3.13} and Propositions \ref{2.23}--\ref{2.24} in hand, we can   easily establish   Theorem \ref{1.12} as follows.
	
	Let the initial data $(\eta^0,u^0)\in H^4_* \times H^3$  and  the elasticity coefficient $\kappa$   satisfy
	\begin{align}
		\max\{K , K^4\}/\kappa\leqslant\min\{\delta_2, \delta_3^2/2,  \delta_4^2/2\},
		\label{2.26}
	\end{align}
	where the positive constant $K$ is defined in \eqref{2.25} with $c_2 \geqslant 1$. Since
	\begin{align}
		K^2 \leqslant \frac{2 K}{3} +\frac{K^4}{3} \leqslant 2\max\{K, K^4\},
		\label{2.27}
	\end{align}
	we can easily obtain from \eqref{2.26} that $\|\nabla \eta^0\|_3 \leqslant \delta_3$. Therefore, according to Proposition \ref{2.23}, the Cauchy problem \eqref{1.21} has a unique local classical solution $(\eta, u)\in C^0([0, T^\mm{max}); H^4) \times\mathfrak{U}_{T^\mm{max}}$ with the maximal  existence of time $T^\mm{max}$ satisfying
	\begin{itemize}
		\item for any $T\in  I_{T^{\max}}$,
		the solution $(\eta,u)$ belongs to $C^0(\overline{I_T};{H}^{4})\times \mathfrak{U}_T$ and $$\sup_{t\in [0,T]} \|\nabla \eta\|_3\leqslant 2\delta_2;$$
		\item $\limsup_{t\to T^{\max} }\|\nabla \eta( t)\|_3 > \delta_2$ or $\limsup_{t\to T^{\max} }\|(\nabla \eta,u)( t)\|_3=\infty$,
		if $T^{\max}<\infty$.
	\end{itemize}
	
	Define
	$$T^\star = \sup \left\{T \in I_{T^\mm{max}} ~\left|~ \mathcal{E}(t) \leqslant K^2 \mbox{ for any } t\leqslant T\right.\right\}.$$
	Then, from the definition of $K$, it is obvious that the value of $T^\star$ as defined above makes sense and $T^\star > 0$. In addion, since $\|\nabla\eta(t)\|_3 \leqslant K/\sqrt{\kappa} \leqslant \delta_4$ for any $t\in (0,T^\star)$, one gets by Proposition \ref{2.24} that $\eta(t) \in H_*^4$ for any $t \in (0,T^\star)$.
	Next it suffices to derive that $T^\star = \infty$.
	
	We assume that  $T^\star = \infty$ fails, then $T^\star < \infty$. Using the inequality $K/\sqrt{\kappa}\leqslant \delta_3$ from \eqref{2.26} and \eqref{2.27}, then applying Proposition \ref{2.23}, we can easily observe from the definitions of $T^\mm{max}$ and $T^\star$ that  $T^\star<T^{\max}$ and
	\begin{equation}
		\label{2.28}
		\mathcal{E}(  T^\star)=K^2.
	\end{equation}
	Noting that $\max\{K,K^4\}/\kappa \leqslant \delta_2$ and ${\sup}_{0 \leqslant t \leqslant T^\star} \mathcal{E}(t)  \leqslant K^2$, by the regularity of $(\eta, u)$, we can deduce that the solution $(\eta, u)$ satisfies the stability estimate \eqref{2.22} for any $t\in I_{T^\star}$. It further can be derived from \eqref{3.13} that, for any $t \in [0, T^\star]$,
	$$\mathcal{E}(t) \leqslant   K^2 / 2,$$
	which is in contradiction with \eqref{2.28}. Hence, $ T^*=\infty$. Therefore we immediately obtain the global-in-time existence of classical solutions with the stability estimate \eqref{0.11} in Theorem \ref{1.12} by taking $c_1:= \min\{\delta_2, \delta_3^2/2, \delta_4^2/2\}$.
	
	Note that $\|\nabla\eta(t)\|_3\leqslant 2\delta_3$ for any $t>0$. The unique solvability of global solutions follows directly from the local one stated in Proposition \ref{2.23}.
	
	\section{Proof of Theorem \ref{1.10}}
	\label{202240110261721}
	Applying the similar arguments used for Theorem \ref{1.12}, we can easily obtain a unique global classical solution  $(\eta^l,u^l)\in C^0(\mathbb{R}_0^+; H^4 \times H^3)$ to the linear Cauchy problem \eqref{1.13}; moreover the linear solution enjoys the following  stability estimate with temporal decay, as well as the nonlinear solution $(\eta,u)$ provided by Theorem \ref{1.12}:
	\begin{align}
		\label{1.8}
		\tilde{\mathcal{E}}(t) + \int_0^t \tilde{\mathcal{D}}(\tau) \mm{d} \tau \lesssim \mathcal{E}^0,
	\end{align}where
	$\tilde{\mathcal{E}}(t)$ and $\tilde{\mathcal{D}}(t)$ are defined by $ {\mathcal{E}}(t)$ and $ {\mathcal{D}}(t)$ with $(\eta^l, u^l)$ in place of  $(\eta, u)$, respectively.
	
	Letting $(\eta^d, u^d) = (\eta, u) -(\eta^l, u^l)$, then the error function $(\eta^d, u^d)$ satisfies
	\begin{equation}
		\label{3.14}
		\begin{cases}
			\eta^d_t =u^d,\\
			u^d_t -\nabla( P'(1)
			\mm{div}\eta^d +\lambda\mm{div} u^d)-\Delta(\mu u^d
			+\kappa\eta^d)= \mathcal{N},\\
			(\eta^d, u^d)|_{t =0} = (0, 0).
		\end{cases}
	\end{equation}
	Keeping in mind that $\bar{\mathcal{E}}^0 =0$, and then
	following the derivation of Lemma \ref{3.12}, we get
	\begin{align}
		\bar{\mathcal{E}}(t) + \int_0^t \bar{\mathcal{D}}(\tau) \mm{d}\tau & \lesssim \int_0^t \left| \left(\mathcal{N} \left|\, u^d -\frac{\mu}{4 }\Delta \eta^d -\frac{\mu}{8 }(\tau +1)\Delta u^d \right. \right)_{H^2} \right| \mm{d} \tau\nonumber\\
		& \quad +\int_0^t \left|\left(\mathcal{N}\left|\,\alpha(\tau +1) \Delta^2 \eta^d +\frac{\alpha}{2}(\tau +1)^2\Delta^2 u^d \right.\right)_{H^1}\right| \mm{d} \tau,
		\label{3.16}
	\end{align}
where
		$\bar{\mathcal{E}}(t)$ and $\bar{\mathcal{D}}(t)$ are defined by $ {\mathcal{E}}(t)$ and $ {\mathcal{D}}(t)$ with $(\eta^d, u^d)$ in place of  $(\eta, u)$, respectively. Since the solutions $(\eta, u)$ and $(\eta^l, u^l)$ enjoy the estimates \eqref{0.11} and \eqref{1.8}, respectively, one can get by \eqref{4.20} and \eqref{4.59} that
	\begin{align}
		& \int_0^t \left(\mathcal{N}\left|\,u^d -\frac{\mu}{4 }\Delta \eta^d -\frac{\mu}{8 }(\tau +1)\Delta u^d \right.\right)_{H^2} \mm{d} \tau\nonumber\\
		& +\int_0^t \left(\mathcal{N} \left|\,\alpha(\tau +1) \Delta^2 \eta^d +\frac{\alpha}{2}(\tau +1)^2\Delta^2 u^d \right.\right)_{H^1} \mm{d} \tau \nonumber\\
		&\lesssim \int_0^t \|\mathcal{N} \|_2 \left( \|(u, u^l)\|_2 + \|\Delta(\eta, \eta^l)\|_2 +( \tau +1)\|\Delta (u, u^l)\|_2 \right.\nonumber\\
		& \quad \left.+ (\tau +1)\|\Delta(\eta, \eta^l)\|_2 + (\tau +1)^2 \|\nabla\Delta(u, u^l)\|_1 \right) \mm{d} \tau \lesssim \kappa^{-\frac{1}{2}}
		\left(\mathcal{E}^\frac{3}{2}_0 +\mathcal{E}_0^3 \right).
		\label{3.15}
	\end{align}
	Inserting \eqref{3.15} into \eqref{3.16} arises the desired estimate \eqref{1.14}.
	This completes the proof of Theorem \ref{1.10}.
	
	\section{Proof of Theorem \ref{1.24}}\label{4}
	
	This section is devoted to completing the proof of Theorem \ref{1.24}. We first present the analysis of the spectral structure of the linear problem \eqref{1.13}, and then obtain the optimal $L^2$-norm decay rates for the linear solution $(\eta^l, u^l)$ of \eqref{1.13} in Section \ref{20241027}. Making use of the decay estimate of $(\eta,u)$ in \eqref{0.11} and the Duhamel's principle, we then further obtain better decay estimates, as well as the ones of the linear solution, for the solution $(\eta, u)$ of the nonlinear problem \eqref{1.21}  in Section \ref{202410272204}.
	
	\subsection{Decay estimates for the linear problem}\label{20241027}
	
	Note that the Fourier transform of the Green matrix
	$$
	G = \begin{bmatrix}
		0 & I_{3 \times 3} \\
		P'(1)\nabla\mathrm{div} +\kappa \Delta & \lambda \nabla \mathrm{div} +\mu \Delta
	\end{bmatrix}
	$$
	corresponding to the equations \eqref{1.13}$_1$ and \eqref{1.13}$_2$ generates a $6 \times 6$ matrix, which increases in the complexities in spectral analysis. Fortunately, this trouble can be effectively addressed by using the Hodge decomposition technique introduced in \cite{MR1779621}. Define the pseudo-differential operator $\Lambda^s$ by
		$$\Lambda^s f := (-\Delta)^{\frac{s}{2}} f = \mathscr{F}^{-1}\big(|\xi|^s \widehat{f}\big) \mbox{ for }s\in \mathbb{R},$$
	where $ \mathscr{F}(f)$ or $\widehat{f}$ represents the Fourier transform of $f$, and $\mathscr{F}^{-1}$ is its inverse. Let
		\begin{align}
			d^l = \Lambda^{-1}\mathrm{div}\eta^l, \, M^l = \Lambda^{-1}\mathrm{curl}\eta^l, \, 
			\mathscr{A}^l =\Lambda^{-1}\mathrm{div}u^l, \, \mbox{ and }
			\mathscr{M}^l = \Lambda^{-1}\mathrm{curl}u^l,
			\label{4.45}
		\end{align}
		where $(\mathrm{curl} z)_{ij}=\partial_{j} z^i -\partial_{i} z^j$ and $(\mathrm{curl} z)_{ij}$ is the $ij$-th entry of the $3\times 3$ matrix, we can decompose the linear problem $\eqref{1.13}$ into the following two linear initial-value problems:
	\begin{equation}\label{4.60}
		\begin{cases}
			d^l_t = \mathscr{A}^l,\\
			\mathscr{A}^l_t -(P'(1) +\kappa)\Delta d^l -(\lambda +\mu)\Delta \mathscr{A}^l  =0,\\
			(d^l, \mathscr{A}^l)|_{t = 0} =(d^l_0, \mathscr{A}^l_0)= (\Lambda^{-1}\mathrm{div}\eta^0, \Lambda^{-1}\mathrm{div}u^0),
		\end{cases}
	\end{equation}
	and
	\begin{equation}\label{4.61}
		\begin{cases}
			M^l_t = \mathscr{M}^l,\\
			\mathscr{M}^l_t -\kappa \Delta M^l  - \mu\Delta \mathscr{M}^l= 0,\\
			(M^l, \mathscr{M}^l)|_{t =0} =(M^l_0, \mathscr{M}^l_0)=(\Lambda^{-1}\mathrm{curl}\eta^0, \Lambda^{-1}\mathrm{curl}u^0).
		\end{cases}
	\end{equation}
	
	Due to the similarity between \eqref{4.60} and \eqref{4.61}, our analysis will be concentrated solely on the problem \eqref{4.60}. Letting
	$$\tilde{U}^l =\begin{bmatrix}  d^l\\ \mathscr{A}^l\end{bmatrix}\mbox{ and }
	\tilde{G} = \begin{bmatrix}
		0 & 1 \\
		( P'(1) +\kappa) \Delta & (\lambda +\mu) \Delta
	\end{bmatrix},
	$$
	then we can reformulate the linear problem \eqref{4.60} as follows
	\begin{equation}\label{4.35}
		\begin{cases}
			\partial_t \tilde{U}^l = \tilde{G} \tilde{U}^l,\\
			\left.\tilde{U}^l\right|_{t =0} = \tilde{U}_0^l.
		\end{cases}
	\end{equation}
	We formally denote by $e^{t\tilde{G}}$ the semigroup generated by $\tilde{G}$, as in \cite{MR4152207, MR4768701}. Thus, the solution of \eqref{4.35} is written as $\tilde{U}(t) =e^{t\tilde{G}} \tilde{U}_0^l$.
	
	To investigate the large time behavior of $\tilde{U}(t) =e^{t\tilde{G}} \tilde{U}_0^l$, we apply the Fourier transform with respect to $x$ to \eqref{4.35}, and then obtain that
	\begin{equation}\label{4.36}
		\begin{cases}
			\partial_t \widehat{\tilde{U}^l} = \tilde{\mathcal{G}}(\xi) \widehat{\tilde{U}^l},\\
			\left.\widehat{\tilde{U}^l}\right|_{t =0} = \widehat{\tilde{U}_0^l},
		\end{cases}
	\end{equation}
	where
	$$
	\tilde{\mathcal{G}}(\xi) := \widehat{\tilde{G}}= \begin{bmatrix}
		0 & 1 \\
		-( P'(1) +\kappa) |\xi|^2 & -(\lambda +\mu) |\xi|^2
	\end{bmatrix}
	$$
	and $\xi =(\xi_1, \xi_2, \xi_3)^\top $.  Therefore, the roots of the characteristic equation
	of the above matrix
	\begin{equation}\label{4.36saf}\gamma^2 +(\lambda +\mu)|\xi|^2 \gamma +(P'(1) +\kappa)|\xi|^2 =0
	\end{equation}
	read as follows:
	\begin{align}
		\gamma_\pm = \left( -(\lambda +\mu) |\xi|^2 \pm {\mm{i}}|\xi|\sqrt{4(P'(1) +\kappa) -(\lambda +\mu)^2|\xi|^2} \right)/2,
		\label{4.44}
	\end{align}
	which are called the characteristic eigenvalues.
	Inspired of \cite{MR4643428}, we derive the following expression for $ e^{t\tilde{\mathcal{G}}(\xi)} \widehat{\tilde{U}_0^l}$.
	\begin{lem}
		\label{4.37}
		The solution of the problem \eqref{4.36} is written as
		\begin{align}
			\widehat{\tilde{U}^l} = e^{t\tilde{\mathcal{G}}(\xi)} \widehat{\tilde{U}_0^l}: =\begin{bmatrix}
				\tilde{\mathcal{G}}_{11}(\xi, t) & \tilde{\mathcal{G}}_{12}(\xi, t)\\
				\tilde{\mathcal{G}}_{21}(\xi, t) & \tilde{\mathcal{G}}_{22}(\xi, t)
			\end{bmatrix}\begin{bmatrix}
				\widehat{d^l_0} \\
				\widehat{\mathscr{A}^l_0}
			\end{bmatrix},
			\label{4.38}
		\end{align}
		where  $\tilde{\mathcal{G}}_{ij}$ for $1\leqslant i$, $j\leqslant 2$ are provided as follows:
		\begin{enumerate}[(1)]
			\item If $|\xi| \neq  {2\sqrt{P'(1) +\kappa}}/({\lambda +\mu})$, i.e. $\gamma_+ \neq \gamma_-$, then
			\begin{align}
				\label{202411101859}
				\begin{cases}
					\tilde{\mathcal{G}}_{11}(\xi, t)  =({\gamma_+ e^{\gamma_-t} -\gamma_- e^{\gamma_+t}})/({\gamma_+ -\gamma_-}),\\
					\tilde{\mathcal{G}}_{12}(\xi, t)  = ({e^{\gamma_+t} -e^{\gamma_-t}})/({\gamma_+ -\gamma_-}),\\
					\tilde{\mathcal{G}}_{21}(\xi, t)  =- (P'(1) +\kappa) |\xi|^2 \tilde{\mathcal{G}}_{12}(\xi, t),  \\
					\tilde{\mathcal{G}}_{22}(\xi, t)  = \tilde{\mathcal{G}}_{11}(\xi, t) -(\lambda +\mu) |\xi| ^2 \tilde{\mathcal{G}}_{12}(\xi, t) .
				\end{cases}
			\end{align}
			\item If $|\xi| = {2\sqrt{P'(1) +\kappa}}/({\lambda +\mu})$, i.e. $\gamma_+ =\gamma_-$, then
			\begin{align}
				\begin{cases}\tilde{\mathcal{G}}_{11}(\xi, t)  =(1 -\gamma_+ t )e^{\gamma_+ t}, \ \tilde{\mathcal{G}}_{12}(\xi, t)  = te^{\gamma_+ t}, \\
					\tilde{\mathcal{G}}_{21}(\xi, t) =-{4(P'(1) +\kappa)^2}{(\lambda +\mu)^{-2}} \tilde{\mathcal{G}}_{12}(\xi, t),\\
					\tilde{\mathcal{G}}_{22}(\xi, t)  = \tilde{\mathcal{G}}_{11}(\xi, t) - {4(P'(1) +\kappa)}( {\lambda +\mu})^{-1}  \tilde{\mathcal{G}}_{12}(\xi, t).
				\end{cases}\label{2024111018591}
			\end{align}
		\end{enumerate}
	\end{lem}
	\begin{pf} We derive the  formula \eqref{4.38} by tow cases.
		
		(1) For $|\xi| \neq {2\sqrt{P'(1) +\kappa}}/({\lambda +\mu})$, i.e. $\gamma_+ \neq \gamma_-$, we can derive there exists an invertible matrix $\mathcal{P}$ such that
		\begin{align}
			\mathcal{P}^{-1} \tilde{\mathcal{G}}(\xi) \mathcal{P} =\begin{bmatrix}
				\gamma_+ & 0\\
				0 & \gamma_-
			\end{bmatrix} =:\mathcal{J}.\nonumber
		\end{align}
		By making tedious calculations, it is clear that
		\begin{align}
			e^{t\tilde{\mathcal{G}}(\xi)} & = e^{\gamma_+ t}\mathcal{P}\begin{bmatrix}
				1 & 0\\
				0 & 0
			\end{bmatrix} \mathcal{P}^{-1} +e^{\gamma_- t}\mathcal{P}\begin{bmatrix}
				0 & 0\\
				0 & 1
			\end{bmatrix} \mathcal{P}^{-1} \nonumber\\
			& =e^{\gamma_+ t} \mathcal{P}_1(\xi) +e^{\gamma_- t} \mathcal{P}_2(\xi).
			\label{A.1}
		\end{align}
		Here, the projectors $\mathcal{P}_i(\xi)$, $i=1, 2$, can be computed as follows
		\begin{align}
			\mathcal{P}_1(\xi) & = \mathcal{P}\begin{bmatrix}
				1 & 0\\
				0 & 0
			\end{bmatrix} \mathcal{P}^{-1} =({\gamma_+ -\gamma_- })^{-1} \mathcal{P}(\mathcal{J} -\gamma_- I_{2\times 2})\mathcal{P}^{-1} =\frac{\tilde{\mathcal{G}}(\xi) -\gamma_-I_{2\times 2}}{\gamma_+ -\gamma_-}\nonumber\\
			& =(\gamma_- -\gamma_+)^{-1}\begin{bmatrix}
				\gamma_- & -1\\
				(P'(1) +\kappa)|\xi|^2 & (\lambda +\mu)|\xi|^2 +\gamma_-   \label{4.66}
			\end{bmatrix}
		\end{align}
		and
		\begin{align}
			\mathcal{P}_2(\xi) & = \mathcal{P}\begin{bmatrix}
				0 & 0\\
				0 & 1
			\end{bmatrix} \mathcal{P}^{-1} = (\gamma_- -\gamma_+ )^{-1} \mathcal{P}(\mathcal{J} -\gamma_+ I_{2\times 2})\mathcal{P}^{-1} =\frac{\tilde{\mathcal{G}}(\xi) -\gamma_+I}{\gamma_- -\gamma_+} \nonumber\\
			& =(\gamma_+ - \gamma_{-} )^{-1}\begin{bmatrix}
				\gamma_+  & - {1} \\
				(P'(1) +\kappa)|\xi|^2  &   (\lambda +\mu)|\xi|^2 +\gamma_+  \label{4.80}
			\end{bmatrix}.
		\end{align}
		Consequently, by the fomulas \eqref{A.1}--\eqref{4.80}, we can represent the solution of the problem \eqref{4.36} as
		\begin{align}
			\widehat{\tilde{U}^l} =e^{t\tilde{\mathcal{G}}(\xi)}\widehat{\tilde{U}_0^l}  =\begin{bmatrix}
				\tilde{\mathcal{G}}_{11}(\xi, t) & \tilde{\mathcal{G}}_{12}(\xi, t)\\
				\tilde{\mathcal{G}}_{21}(\xi, t) & \tilde{\mathcal{G}}_{22}(\xi, t)
			\end{bmatrix}\begin{bmatrix}
				\widehat{d^l_0} \\
				\widehat{\mathscr{A}^l_0}
			\end{bmatrix},\nonumber
		\end{align}
		where $\tilde{\mathcal{G}}_{ij}(\xi, t) $ are defined by \eqref{202411101859} for $1\leqslant i$, $j\leqslant  2$.
		
		(2) For $|\xi| = {2\sqrt{P'(1) +\kappa}}/({\lambda +\mu})$, i.e. $\gamma_+ = \gamma_-$, we can derive there exists an invertible matrix $\tilde{\mathcal{P}}$ such that
		\begin{align}
			\tilde{\mathcal{P}}^{-1} \tilde{\mathcal{G}}(\xi) \tilde{\mathcal{P}} =\begin{bmatrix}
				\gamma_+ & 1\\
				0 & \gamma_+
			\end{bmatrix} =: \tilde{\mathcal{J}}.\nonumber
		\end{align}
		By making tedious calculations, it is clear that
		\begin{align}
			e^{t\tilde{\mathcal{G}}(\xi)} & = e^{\gamma_+ t}\tilde{\mathcal{P}}\begin{bmatrix}
				1 & 0\\
				0 & 1
			\end{bmatrix} \tilde{\mathcal{P}}^{-1} +te^{\gamma_+ t}\tilde{\mathcal{P}}\begin{bmatrix}
				0 & 1\\
				0 & 0
			\end{bmatrix} \tilde{\mathcal{P}}^{-1} \nonumber\\
			& =e^{\gamma_+ t} I_{2\times 2} +te^{\gamma_+ t} \tilde{\mathcal{P}_1}(\xi). \label{A.2}
		\end{align}
		Here, the projector
		\begin{align}
			\tilde{\mathcal{P}_1}(\xi) & =\tilde{P} \begin{bmatrix}
				0 & 1\\
				0 & 0
			\end{bmatrix}\tilde{P}^{-1} = \tilde{P}(\tilde{\mathcal{J}} -\gamma_{+} I)\tilde{P}^{-1} = \tilde{\mathcal{G}}(\xi) -\gamma_{+}  I_{2\times 2}  \nonumber\\
			& =-\begin{bmatrix}
				\gamma_+ & -1\\
				{4(P'(1) +\kappa)^2}{( \lambda +\mu )^{-2}} & \gamma_{+} +{4(P'(1) +\kappa)}/({\lambda +\mu})
			\end{bmatrix}.\label{4.88}
		\end{align}
		Consequently, by the formulas \eqref{A.2} and \eqref{4.88}, we can represent the solution of the problem \eqref{4.36} as follows
		\begin{align}
			\widehat{\tilde{U}^l} =e^{t\tilde{\mathcal{G}}(\xi)}\widehat{\tilde{U}_0^l} =\left(e^{\gamma_+ t}  I_{2\times 2}  +te^{\gamma_+ t}\right) \widehat{\tilde{U}_0^l} =\begin{bmatrix}
				\tilde{\mathcal{G}}_{11}(\xi, t) & \tilde{\mathcal{G}}_{12}(\xi, t)\\
				\tilde{\mathcal{G}}_{21}(\xi, t) & \tilde{\mathcal{G}}_{22}(\xi, t)
			\end{bmatrix}\begin{bmatrix}
				\widehat{d^l_0} \\
				\widehat{\mathscr{A}^l_0}
			\end{bmatrix},\nonumber
		\end{align}
		where $\tilde{\mathcal{G}}_{ij}(\xi, t) $ are defined by \eqref{2024111018591} for $1\leqslant i$, $j\leqslant  2$. This completes the proof. \hfill $\Box$
	\end{pf}
	
	We also need to establish the following asymptotic expressions for the characteristic eigenvalues $\gamma_\pm $.
	\begin{lem}
		\label{4.65}
		\begin{enumerate}[(1)]
			\item If $|\xi| <  {2\sqrt{P'(1) +\kappa}}/({\lambda +\mu})$,  the characteristic eigenvalues  can be represented by the following Taylor series expansion:
			\begin{equation}\label{4.62}
				\begin{cases}
					\gamma_+ = -( {\lambda +\mu} )|\xi|^2/2 +\mm{i}\big(\sqrt{P'(1) +\kappa} |\xi| +O(|\xi|^3)\big), \\
					\gamma_- = - ({\lambda +\mu})|\xi|^2/{2} -\mm{i}\big(\sqrt{P'(1) +\kappa} |\xi| +O(|\xi|^3)\big).
				\end{cases}
			\end{equation}
			\item If $|\xi| >  {2\sqrt{P'(1) +\kappa}}/({\lambda +\mu})$, the characteristic eigenvalues  can be represented by the following Laurent series expansion:
			\begin{equation}\label{4.64}
				\begin{cases}
					\gamma_+ = -(P'(1) +\kappa)(\lambda +\mu)^{-1} +O(|\xi|^{-2}),\\
					\gamma_- = -(\lambda +\mu)|\xi|^2 +(P'(1) +\kappa)(\lambda +\mu)^{-1} +O(|\xi|^{-2}).
				\end{cases}
			\end{equation}
			\item If $|\xi| = 2\sqrt{P'(1) +\kappa}/({\lambda +\mu})$, the characteristic equation has a double real root
			\begin{align}\label{4.63}
				\gamma_+ =\gamma_- = -2(P'(1) +\kappa)/({\lambda +\mu}).
			\end{align}
		\end{enumerate}
	\end{lem}
	\begin{proof}
		Note that the function $f(x) =\sqrt{1 -x}$ can be expended into the following Taylor series
		$$f(x)=1 -x/2 - x^2/8 +O(x^3)\mbox{ for any } x\in(-1, 1).$$
		
		(1) For $|\xi| < {2\sqrt{P'(1) +\kappa}}/({\lambda +\mu})$, we can obtain that
		\begin{align}
			\gamma_\pm & = -({\lambda +\mu}) |\xi|^2/2 \pm \mm{i}\sqrt{P'(1) +\kappa}|\xi| \sqrt{1 - {(\lambda +\mu)^2}|\xi|^2/{4(P'(1) +\kappa)} }\nonumber\\
			& = -( {\lambda +\mu}) |\xi|^2 /{2}\pm \mm{i}\sqrt{P'(1) +\kappa}|\xi| f\left({(\lambda +\mu)^2}({4(P'(1) +\kappa}))^{-1} |\xi|^2\right) \nonumber\\
			& = - ({\lambda +\mu}) |\xi|^2/2 \pm \mm{i} \sqrt{P'(1) +\kappa}|\xi| \left( 1 -(\lambda +\mu)^2({8(P'(1) +\kappa)})^{-1}|\xi|^2 + O(|\xi|^4)\right)\nonumber\\
			&=  -({\lambda +\mu})|\xi|^2 /2\pm \mm{i}\left(\sqrt{P'(1) +\kappa}|\xi| +O(|\xi|^3)\right),\nonumber
		\end{align}
		which implies the two expansions in \eqref{4.62}.
		
		(2) For $|\xi| >  {2\sqrt{P'(1) +\kappa}}/({\lambda +\mu})$, then
		\begin{align}
			\gamma_\pm & = -( {\lambda +\mu} )|\xi|^2/{2} \pm |\xi|\sqrt{(\lambda +\mu)^2|\xi|^2 -4(P'(1) +\kappa)}/2\nonumber\\
			& = -(\lambda +\mu)|\xi|^2/{2}  \pm {(\lambda +\mu)}|\xi|^2 \sqrt{1 - {4(P'(1) +\kappa)}{(\lambda +\mu)^{-2} |\xi|^{-2}}}/{2}. \nonumber\\
			& = - ({\lambda +\mu}) |\xi|^2/{2} \pm {(\lambda +\mu)}|\xi|^2 f\left( {4(P'(1) +\kappa)}{(\lambda +\mu)^{-2} |\xi|^{-2}}\right)/2\nonumber\\
			& = - ({\lambda +\mu})|\xi|^2/2  \pm  {(\lambda +\mu)}|\xi|^2 \left(1 -{2(P'(1) +\kappa)}{(\lambda +\mu)^{-2}|\xi|^{-2}}\right.\nonumber\\
			&\quad \left. - 2(P'(1) +\kappa)^2 (\lambda +\mu)^{-4} |\xi|^{-4} +O(|\xi|^{-6})\right)/2\nonumber\\
			& = -({\lambda +\mu})|\xi|^2 /{2} \pm  \left({(\lambda +\mu)}|\xi|^2/2 -({P'(1) +\kappa})({\lambda +\mu})^{-1} \right.\nonumber\\
			&\quad \left. -(P'(1) +\kappa)^2 (\lambda +\mu)^{-3}|\xi|^{-2} +O(|\xi|^{-6})\right) \nonumber\\
			& = -(\lambda +\mu) |\xi|^2/2 \pm  \left({(\lambda +\mu)}|\xi|^2/2 -({P'(1) +\kappa})({\lambda +\mu})^{-1}  +O(|\xi|^{-2})\right), \nonumber
		\end{align}
		which implies the asymptotic expansions in \eqref{4.64}.
		
		(3) For $|\xi| =  {2\sqrt{P'(1) +\kappa}}/({\lambda +\mu})$, then the formula \eqref{4.63} can be derived through straightforward computations.
	\end{proof}
	With the aid of Lemmas \ref{4.37} and \ref{4.65}, we can further derive the large time behavior for $\widehat{\tilde{U}^l}$.
	\begin{lem}\label{4.72}
		The solution $\widehat{\tilde{U}^l} =(\widehat{d^l}, \widehat{\mathscr{A}^l})$ to the linear problem \eqref{4.36} enjoys that
		\begin{enumerate}[(1)]
			\item If $|\xi| < {2\sqrt{P'(1) +\kappa} }/({\lambda +\mu})$, then
			\begin{align}
				\left|\widehat{\tilde{U}^l}\right| \lesssim e^{-\frac{\lambda +\mu}{4}|\xi|^2 t} \left( \left|\widehat{d^l_0}\right| +  \left|\widehat{\mathscr{A}^l_0}\right| \right).
				\label{4.69}
			\end{align}
			\item If $|\xi| > {2\sqrt{P'(1) +\kappa}} /({\lambda +\mu})$, then
			\begin{align}
				\left|\widehat{d^l}\right| & \lesssim \big( 1 +|\xi|^{-2} \big)e^{-R t}  \left|\widehat{d^l_0}\right| +|\xi|^{-2}e^{-Rt} \left|\widehat{\mathscr{A}^l_0} \right|,\label{4.73}\\
				\left|\widehat{\mathscr{A}^l}\right| & \lesssim e^{-Rt} \left|\widehat{d^l_0}\right| + \big( 1 +|\xi|^{-2} \big) e^{-Rt}  \left|\widehat{\mathscr{A}^l_0}\right|. \label{4.40}
			\end{align}
			\item If $|\xi| =  {2\sqrt{P'(1) +\kappa}} /({\lambda +\mu})$, then
			\begin{align}\label{4.70}
				\left|\widehat{\tilde{U}^l}\right | \lesssim e^{-Rt} \left(  \left|\widehat{d^l_0}\right| + \left|\widehat{\mathscr{A}^l_0}\right| \right).
			\end{align}
		\end{enumerate}
Here, $R$ is a positive constant that only depends on the parameters $\lambda$, $\mu$ and $\kappa$.
	\end{lem}
	\begin{proof}We derive the above three conclusions in Lemma \ref{4.72} in sequence.
		
		(1) If $|\xi| < {2\sqrt{P'(1) +\kappa}} /({\lambda +\mu})$, by the expressions \eqref{4.38} and \eqref{4.62}, we have
		\begin{align}
			\widehat{d^l} & = \tilde{\mathcal{G}}_{11}(\xi, t) \widehat{d^l_0} +\tilde{\mathcal{G}}_{12}(\xi, t) \widehat{\mathscr{A}^l_0} \nonumber\\
			& = e^{- (\lambda +\mu)|\xi|^2 t/2} \left(\cos(\theta_1 t) +\theta_2 \sin(\theta_1 t) \right)\widehat{d^l_0} + {\theta_1^{-1} }{\sin(\theta_1 t)} e^{- (\lambda +\mu)|\xi|^2 t/{2}} \widehat{\mathscr{A}^l_0}, \label{4.67} \\
			\widehat{\mathscr{A}^l} & = \tilde{\mathcal{G}}_{21}(\xi, t) \widehat{d^l_0} + \tilde{\mathcal{G}}_{22}(\xi, t) \widehat{\mathscr{A}^l_0} \nonumber\\
			& = -{\theta_1^{-1}}\left( P'(1) +\kappa \right) |\xi|^2  {\sin(\theta_1 t)} e^{-(\lambda +\mu)|\xi|^2 t/2} \widehat{d^l_0} \nonumber\\
			& \quad +e^{- (\lambda +\mu)|\xi|^2 t/{2}} \left( \cos(\theta_1 t) +\theta_2 \sin(\theta_1 t) -\left( \lambda +\mu \right) {\theta_1^{-1}}|\xi|^2  {\sin(\theta_1 t)} \right) \widehat{\mathscr{A}^l_0}, \label{4.68}
		\end{align}
		where we have defined that
		$$\theta_1: = \sqrt{P'(1) +\kappa}|\xi| +O(|\xi|^3)\mbox{ and }\theta_2: =  {(\lambda +\mu) |\xi|^2}/({2 \sqrt{P'(1) +\kappa}|\xi| +O(|\xi|^3)}).$$ Therefore, the estimate \eqref{4.69} can be easily derived from \eqref{4.67} and \eqref{4.68}.
		
		(2) If $|\xi| > {2\sqrt{P'(1) +\kappa}} /({\lambda +\mu})$,  by the expressions \eqref{4.38}, it holds that
		\begin{align}
			\widehat{d^l} & = \tilde{\mathcal{G}}_{11}(\xi, t) \widehat{d^l_0} +\tilde{\mathcal{G}}_{12}(\xi, t) \widehat{\mathscr{A}^l_0} \nonumber\\
			&=(\gamma_+ -\gamma_-)^{-1}  \left( \left( {\gamma_+} e^{\gamma_- t} - {\gamma_-} e^{\gamma_+ t}\right)\widehat{d^l_0} +  \left( e^{\gamma_+ t} - e^{\gamma_- t}\right) \widehat{\mathscr{A}^l_0}\right),\label{4.71}\\
			\widehat{\mathscr{A}^l} & = \tilde{\mathcal{G}}_{21}(\xi, t) \widehat{d^l_0} +\tilde{\mathcal{G}}_{22}(\xi, t) \widehat{\mathscr{A}^l_0} \nonumber\\
			&= {(P'(1) +\kappa)|\xi|^2}({\gamma_+ -\gamma_-})^{-1} \left( e^{\gamma_- t} - e^{\gamma_+ t} \right) \widehat{d^l_0} \nonumber\\
			& \quad +(\gamma_+ -\gamma_-)^{-1} \left(\gamma_+ e^{\gamma_- t} - \gamma_- e^{\gamma_+ t} +{(\lambda +\mu)|\xi|^2}\left( e^{\gamma_- t} - e^{\gamma_+ t} \right)\right)\widehat{\mathscr{A}^l_0}.\label{4.39}
		\end{align}
		With the help of expressions of $\gamma_\pm$ in \eqref{4.64}, we have
		$$ \left|{\gamma_+}({\gamma_+ -\gamma_-})^{-1}\right|, \quad
		\left|({\gamma_+ -\gamma_-})^{-1}\right| \lesssim |\xi|^{-2}, \quad
		\left|{\gamma_-}({\gamma_+ -\gamma_-})^{-1}\right| \lesssim 1 \mbox{ and } |e^{\gamma_\pm t}| \leqslant e^{-Rt}$$
		for some constant $R >0$. Therefore,  \eqref{4.73} and \eqref{4.40} can be easily derived from \eqref{4.71} and \eqref{4.39}, respectively.
		
		(3) If $|\xi| =  {2\sqrt{P'(1) +\kappa}} /({\lambda +\mu})$, by both expressions \eqref{4.38} and \eqref{4.63}, we deduce that
		\begin{align}
			\widehat{d^l} & = \left( 1 + {2(P'(1) +\kappa)}({\lambda +\mu})^{-1} t\right) e^{-{2(P'(1) +\kappa)}t/({\lambda +\mu}) } \widehat{d^l_0} +te^{- {2(P'(1) +\kappa)}t/({\lambda +\mu})} \widehat{\mathscr{A}^l_0},\\
			\widehat{\mathscr{A}^l} & =  - {4(P'(1) +\kappa)^2}{(\lambda +\mu)^{-2}}te^{-{2(P'(1) +\kappa)}t/(\lambda +\mu)} \widehat{d^l_0} \nonumber \\
			&\quad +\left( 1 - {2(P'(1) +\kappa)}{(\lambda +\mu)^{-1}} t\right) e^{-2(P'(1) +\kappa)t/(\lambda +\mu)} \widehat{\mathscr{A}^l_0},
		\end{align}
		which result in \eqref{4.70}. This completes the proof of Lemma \ref{4.72}.
	\end{proof}
	
	Thanks to Lemma \ref{4.65}, we are now equipped to establish the  temporal decay rates of the linear solution
	$$\tilde{U^l} = e^{t\tilde{G}} \tilde{U}_0^l:=\mathscr{F}^{-1}\left(e^{t \mathcal{\tilde{G}}_\xi} \widehat{\tilde{U_0^l}} \right).$$
	\begin{lem}
		\label{4.82}
		Assume $ (d^l_0, \mathscr{A}^l_0) \in H^{j}\times H^{j}$ with $j\geqslant 2$ and $ \|(d^l_0, \mathscr{A}^l_0)\|_{L^1}$ is bounded, then the solution $\tilde{U^l} =(d^l, \mathscr{A}^l)$ of the linear problem \eqref{4.35} satisfies that
		\begin{align}
			\|\nabla^k (d^l, \mathscr{A}^l) \|_0 ^2 & \lesssim (1 +t)^{-\frac{3}{2} -k} \big( \|(d^l_0, \mathscr{A}^l_0)\|_{L^1}^2 +\|\nabla^k (d^l_0, \mathscr{A}^l_0)\|_0^2 \big), \label{4.77}
		\end{align}
		where $k =0$, $1$ and $2$.
	\end{lem}
	\begin{proof}
		From Lemma \ref{4.72},   Plancherel theorem, and Hausdorff--Young's inequality, it holds for $0\leqslant k \leqslant 2$ that
		\begin{align}
			\|\nabla^k d^l\|_0^2  = \||\xi|^k \widehat{d^l}\|_0^2
			& \lesssim \int_{|\xi| < \frac{2\sqrt{P'(1) +\kappa}}{\lambda +\mu}} |\xi|^{2k} e^{-(\lambda +\mu)|\xi|^2 t/2} \big( |\widehat{d^l_0}|^2 +|\widehat{\mathscr{A}^l_0}|^2 \big) \mm{d}\xi \nonumber\\
			&\quad + \int_{ |\xi| \geqslant \frac{2\sqrt{P'(1) +\kappa}}{\lambda +\mu} } |\xi|^{2k} e^{-2Rt} \big( 1 +|\xi|^{-4} \big) \big( |\widehat{d^l_0} |^2 +| \widehat{\mathscr{A}^l_0} |^2 \big)\mm{d}\xi \nonumber\\
			& \lesssim (1 +t)^{-\frac{3}{2} -k} \| (d^l_0, \mathscr{A}^l_0) \|_{L^1}^2 +e^{-2Rt} \| \nabla^k (d^l_0, \mathscr{A}^l_0) \|_0^2 \nonumber\\
			& \lesssim (1 +t)^{-\frac{3}{2} -k} \left( \| (d^l_0, \mathscr{A}^l_0) \|_{L^1}^2 +\|\nabla^k (d^l_0, \mathscr{A}^l_0) \|_0^2 \right),\nonumber
		\end{align}
		which implies the estimate \eqref{4.77} for $d^l$. In addition, following the similar arguments, we can deduce that the   temporal decay estimates \eqref{4.77} also holds for $\mathscr{A}^l$.
	\end{proof}
	
	For the linear problem \eqref{4.61}, we can compute out the corresponding characteristic eigenvalues
	\begin{align}
		\beta_\pm = \left(-{\mu} |\xi|^2 \pm {\mm{i}}|\xi|\sqrt{4\kappa -\mu^2|\xi|^2} \right)/2.
	\end{align}
	Similarly to Lemma \ref{4.37}, we have
	\begin{lem}
		It holds that
		\begin{align}
			\begin{bmatrix} \widehat{M^l}\\ \widehat{\mathscr{M}^l}
			\end{bmatrix}  =:\widehat{\bar{U}^l} = e^{t\bar{\mathcal{G}}_\xi} \widehat{\bar{U}_0^l}: =\begin{bmatrix}
				\bar{\mathcal{G}}_{11}(\xi, t) I& \bar{\mathcal{G}}_{12}(\xi, t) I\\
				\bar{\mathcal{G}}_{21}(\xi, t) I& \bar{\mathcal{G}}_{22}(\xi, t) I
			\end{bmatrix}\begin{bmatrix}
				\widehat{M^l_0} \\
				\widehat{\mathscr{M}^l_0}
			\end{bmatrix},
			\label{4.41}
		\end{align}
		where  $\bar{\mathcal{G}}_{ij}$ for $1\leqslant i$, $j\leqslant 2$ are provided as follows:
		\begin{enumerate}[(1)]
			\item If $|\xi| \neq {2\sqrt{\kappa}}/{\mu}$, i.e. $\beta_+ \neq \beta_-$, then
			\begin{align}
				\begin{cases}
					\bar{\mathcal{G}}_{11}(\xi, t)  := ({\beta_+ e^{\beta_-t} -\beta_- e^{\beta_+t}})/({\beta_+ -\beta_-}),  \\
					\bar{\mathcal{G}}_{12}(\xi, t) : = ({e^{\beta_+t} -e^{\beta_-t}})/({\beta_+ -\beta_-}), \\
					\bar{\mathcal{G}}_{21}(\xi, t) : =- \kappa |\xi|^2 \bar{\mathcal{G}}_{12}(\xi, t),\
					\bar{\mathcal{G}}_{22}(\xi, t) : = \bar{\mathcal{G}}_{11}(\xi, t) -\mu |\xi| ^2 \bar{\mathcal{G}}_{12}(\xi, t) . \nonumber
				\end{cases}
			\end{align}
			\item If $|\xi| =  {2\sqrt{\kappa}}/{\mu}$, i.e. $\beta_+ =\beta_-$, then
			\begin{align}
				\begin{cases}
					\bar{\mathcal{G}}_{11}(\xi, t)   :=(1 -\beta_+ t )e^{\beta_+ t}, \
					\bar{\mathcal{G}}_{12}(\xi, t) := te^{\beta_+ t},  \\
					\bar{\mathcal{G}}_{21}(\xi, t)  : =-{4\kappa^2}{\mu^{-2}} \bar{\mathcal{G}}_{12}(\xi, t), \\
					\bar{\mathcal{G}}_{22}(\xi, t)   := \bar{\mathcal{G}}_{11}(\xi, t) -{4\kappa}{\mu}^{-1}  \bar{\mathcal{G}}_{12}(\xi, t) .\end{cases}\nonumber
			\end{align}
		\end{enumerate}
	\end{lem}
	Similarly to Lemma \ref{4.82}, we have following  temporal decay rates for the solution
	$$\bar{U^l} = e^{t\bar{G}} \bar{U}_0^l: =\mathscr{F}^{-1}\left(e^{t \mathcal{\bar{G}}_\xi} \widehat{\bar{U_0^l}} \right)\mbox{ with }\bar{U}_0^l =\begin{bmatrix}M^l_0\\ \mathscr{M}^l_0
	\end{bmatrix} . $$
	\begin{lem}
		\label{4.81}
		Assume $(M^l_0, \mathscr{M}^l_0) \in H^j \times H^j$ with $j\geqslant 2$ and $\|(M^l_0, \mathscr{M}^l_0)\|_{L^1}$ is bounded, then the solution $(M^l, \mathscr{M}^l)$ of the linear system \eqref{4.61} satisfies that
		\begin{align}
			\|\nabla^k (M^l, \mathscr{M}^l) \|_0^2 & \lesssim (1 +t)^{-\frac{3}{2} -k} \big( \|(M^l_0, \mathscr{M}^l_0)\|_{L^1}^2 +\|\nabla^k (M^l_0, \mathscr{M}^l_0)\|_0^2 \big), \label{4.79}
		\end{align}
		where $k =0$, $1$ and $2$.
	\end{lem}
	
	Finally, we will present the explicit expressions and the temporal decay rates for the solutions of the linear problem \eqref{1.13}. Obviously we can reformulate \eqref{1.13} as the following equivalent form:
	\begin{equation}\label{4.32}
		\begin{cases}
			U^l_t = G U^l,\\
			U^l|_{t=0}= U^l_0
		\end{cases}
		\mbox{where } U^l :=\begin{bmatrix}
			\eta^l\\ u^l\end{bmatrix}\mbox{ and } U^l_0 : =\begin{bmatrix}\eta^0\\ u^0 \end{bmatrix}.
	\end{equation}
	\begin{pro}
		\label{4.8}
		Assume $U^l_0 \in  H^j$ with $j \geqslant 2$ and $\|U^l_0 \|_{L^1}$ is bounded,
		the solution of \eqref{4.32} is written as
		\begin{align}
			\label{2022411101925}
			U^l(t) = e^{tG} U^l_0 := \mathscr{F}^{-1}\left(e^{t\mathcal{G}(|\xi|)}\widehat{U^l_0}\right),
		\end{align}
		where we have defined that
		$$e^{t\mathcal{G}(|\xi|)}:= \begin{bmatrix}
			\mathcal{G}_{11}(\xi, t) & \mathcal{G}_{12}(\xi, t)\\
			\mathcal{G}_{21}(\xi, t) & \mathcal{G}_{22}(\xi, t)
		\end{bmatrix}
		$$
		with
		\begin{equation}
			\begin{cases}
				\mathcal{G}_{11}(\xi, t)  := \left( \tilde{\mathcal{G}}_{11}(\xi, t) -\bar{\mathcal{G}}_{11}(\xi, t)\right) |\xi|^{-2} \xi \otimes \xi +\bar{\mathcal{G}}_{11} (\xi, t) I, \nonumber\\
				\mathcal{G}_{12}(\xi, t)  := \left(\tilde{\mathcal{G}}_{12}(\xi, t) -\bar{\mathcal{G}}_{12} (\xi, t)\right) |\xi|^{-2} \xi \otimes \xi +\bar{\mathcal{G}}_{12} (\xi, t) I, \nonumber\\
				\mathcal{G}_{21}(\xi, t)  :=\left(\tilde{\mathcal{G}}_{21}(\xi, t) - \bar{\mathcal{G}}_{21}(\xi, t)\right) |\xi|^{-2} \xi \otimes \xi +\bar{\mathcal{G}}_{21} (\xi, t) I, \nonumber\\
				\mathcal{G}_{22}(\xi, t)  := \left( \tilde{\mathcal{G}}_{22}(\xi, t) - \bar{\mathcal{G}}_{22}(\xi, t) \right) |\xi|^{-2} \xi \otimes \xi +\bar{\mathcal{G}}_{22} (\xi, t) I.\nonumber
			\end{cases}
		\end{equation}Moreover,
		\begin{align}
			\| \nabla^k U^l(t) \|_0^2 \lesssim (1 +t)^{-\frac{3}{2} -k} \big( \| U^l_0 \|_{L^1}^2  + \| \nabla^k  U^l_0  \|_0^2 \big),
			\label{4.83}
		\end{align}
		where $k =0$, $1$ and $2$.
	\end{pro}
	\begin{proof}
By both the relations
			$(- \Delta )^{-1} \eta^l = \Lambda^{-2} \eta^l$ and $\Delta \eta^l = \nabla \mm{div} \eta^l + \mm{div}(\mm{curl} \eta^l) $, it holds that
			\begin{align}
				\eta^l & = -(-\Delta)^{-1} \left(\nabla \mm{div} \eta^l + \mm{div}(\mm{curl} \eta^l) \right) = -\Lambda^{-2} \left(\nabla \mm{div} \eta^l + \mm{div}(\mm{curl} \eta^l) \right) \nonumber\\
				& = -\Lambda^{-1} \nabla (\Lambda^{-1} \mm{\div} \eta^l) -\Lambda^{-1} \mm{div} (\Lambda^{-1} \mm{curl} \eta^l) = -\Lambda^{-1} \nabla d^l -\Lambda^{-1} \mathrm{div} M^l.\label{4.42}
			\end{align}
			Since
			\begin{equation}
				\begin{cases}
					\widehat{d^l} = \mm{i} |\xi|^{-1} \xi^\top \widehat{\eta^l}, \  \widehat{M^l} = \mm{i} |\xi|^{-1} \left(\widehat{\eta^l}\xi^\top -\xi\widehat{\eta^l}^\top \right), \\
					\widehat{\mathscr{A}^l} = \mm{i} |\xi|^{-1} \xi^\top \widehat{u^l}, \
					\widehat{\mathscr{M}^l} = \mm{i} |\xi|^{-1} \left(\widehat{u^l}\xi^\top -\xi\widehat{u^l}^\top \right), \label{4.34}
				\end{cases}
		\end{equation}
		then we can derive from \eqref{4.38}, \eqref{4.41}, \eqref{4.42}  and  \eqref{4.34} that
	\begin{align}
				\widehat{ \eta^l } & = -\mm{i}|\xi|^{-1} \widehat{d^l}\xi -\mm{i}|\xi|^{-1}\widehat{M^l} \xi \nonumber\\
				&= -\mm{i}|\xi|^{-1} \left( \tilde{\mathcal{G}}_{11}(\xi, t)\widehat{d^l_0}+\tilde{\mathcal{G}}_{12}(\xi, t)\widehat{\mathscr{A}^l_0} \right)\xi -\mm{i}|\xi|^{-1} \left(  \bar{\mathcal{G}}_{11}(\xi, t)\widehat{M^l_0} + \bar{\mathcal{G}}_{12}(\xi, t)\widehat{\mathscr{M}^l_0} \right)\xi  \nonumber\\
				& = \left(\tilde{\mathcal{G}}_{11}(\xi, t) {|\xi|^{-2}}  {\xi \otimes \xi} +\bar{\mathcal{G}}_{11}(\xi, t) \left( I- {|\xi|^{-2}} {\xi \otimes \xi} \right) \right)\widehat{\eta^0}, \nonumber\\
				& \quad +\left(\tilde{\mathcal{G}}_{12} (\xi, t) {|\xi|^{-2}} {\xi \otimes \xi} +\bar{\mathcal{G}}_{12} (\xi, t) \left({ I - |\xi|^{-2}} {\xi \otimes \xi}  \right)\right) \widehat{u^0}. \label{20224111019251}
		\end{align}
		Similarly, with the aid of the relation 
			\begin{align}
				\label{4.43}
				u^l  =  -\Lambda^{-1} \nabla \mathscr{A}^l -\Lambda^{-1} \mathrm{div} \mathscr{M}^l,
		\end{align}
	it can be obtained from the three formulas \eqref{4.38}, \eqref{4.41} and \eqref{4.34} that
			\begin{align}
				\widehat{u^l} & = \left( \tilde{\mathcal{G}}_{21}(\xi, t){|\xi|^{-2}}{\xi \otimes \xi} +\bar{\mathcal{G}}_{21}(\xi, t)\left({I - |\xi|^{-2}}{\xi \otimes \xi}  \right)\right)\widehat{\eta^0} \nonumber\\
				& \quad +\left( \tilde{\mathcal{G}}_{22}(\xi, t){|\xi|^{-2}}{\xi \otimes \xi} +\bar{\mathcal{G}}_{22}(\xi, t)\left( I - |\xi|^{-2} \xi \otimes \xi   \right)\right) \widehat{u^0}.\label{20220411102125}
		\end{align}
		Consequently, we immediately get \eqref{2022411101925} from \eqref{20224111019251} and
		\eqref{20220411102125}.
		
	Let $1\leqslant j\leqslant 3$ and $0 \leqslant k \leqslant 2$. Recalling Lemmas \ref{4.37} and \ref{4.82}, the linear problem \eqref{4.35} with the initial data $(\eta^0_j, u^0_j)^{\top}$ in place of  $(d_0^l, \mathscr{A}_0^l)^{\top}$ admits a unique solution, denoted by $(\phi_j^l, \psi_j^l)^\top$, which satisfies 
			\begin{align}
				&\|\nabla^k (\phi^l_j, \psi^l_j) \|_0 ^2 \lesssim (1 +t)^{-\frac{3}{2} -k} \left( \|(\eta^0_j, u^0_j)\|_{L^1}^2 +\|\nabla^k (\eta^0_j, u^0_j)\|_0^2 \right), \label{4saf77} \\
				&\widehat{\phi_j^l} = \tilde{\mathcal{G}}_{11}(\xi, t) \widehat{\eta^0_j}+  \tilde{\mathcal{G}}_{12}(\xi, t) \widehat{u^0_j}   
				\mbox{ and }
				\widehat{\psi_j^l} = \tilde{\mathcal{G}}_{21}(\xi, t) \widehat{\eta^0_j}+  \tilde{\mathcal{G}}_{22}(\xi, t) \widehat{u^0_j}. \label{4saf78}
			\end{align} 
			By virtue of \eqref{4saf77}, \eqref{4saf78} and the Plancherel theorem, we have
			\begin{align}
				&\left\| |\xi|^k \left( \tilde{\mathcal{G}}_{11}(\xi, t) \widehat{\eta^0_j}+  \tilde{\mathcal{G}}_{12}(\xi, t) \widehat{u^0_j}   \right) \right\|_0^2+
				\left\| |\xi|^k \left( \tilde{\mathcal{G}}_{21}(\xi, t) \widehat{\eta^0_j}+  \tilde{\mathcal{G}}_{22}(\xi, t) \widehat{u^0_j}  \right) \right\|_0^2\nonumber\\
				& \lesssim \|\nabla^k (\phi_j, \psi_j) \|_0 ^2 \lesssim (1 +t)^{-\frac{3}{2} -k} \left( \|(\eta^0_j, u^0_j)\|_{L^1}^2 +\|\nabla^k (\eta^0_j, u^0_j)\|_0^2 \right). \label{4.90}
			\end{align}
			Similarly to \eqref{4.90}, we can obtain that 
			\begin{align}
				&\left\| |\xi|^k \left( \bar{\mathcal{G}}_{11}(\xi, t) \widehat{\eta^0_j}+  \bar{\mathcal{G}}_{12}(\xi, t) \widehat{u^0_j}  \right) \right\|_0^2+
				\left\| |\xi|^k \left( \bar{\mathcal{G}}_{21}(\xi, t) \widehat{\eta^0_j} +  \bar{\mathcal{G}}_{22}(\xi, t) \widehat{u^0_j} \right) \right\|_0^2 \nonumber\\
				&  \lesssim (1 +t)^{-\frac{3}{2} -k} \left( \|(\eta^0_j, u^0_j)\|_{L^1}^2 +\|\nabla^k (\eta^0_j, u^0_j)\|_0^2 \right). \label{4.91}	
			\end{align}
			From the expression \eqref{20224111019251}, the above two estimates and the fact that $|\xi|^{-2}|\xi_i| |\xi_j| \lesssim 1 $ for any $\xi \in \mathbb{R}^3 $ and $1 \leqslant i$, $j \leqslant 3$, it holds 
			\begin{align}
				\|\nabla^k \eta^l \|_0^2 & \lesssim \left\| |\xi|^{-2} \xi \otimes \xi |\xi|^k \left( \tilde{\mathcal{G}}_{11}(\xi, t) \widehat{\eta^0} + \tilde{\mathcal{G}}_{12}(\xi, t) \widehat{u^0} \right) \right\|_0^2 \nonumber\\
				& \quad +\left\| \left( I - |\xi|^{-2} \xi \otimes \xi \right) |\xi|^k \left( \bar{\mathcal{G}}_{11}(\xi, t) \widehat{\eta^0} + \bar{\mathcal{G}}_{12}(\xi, t) \widehat{u^0} \right) \right\|_0^2 \nonumber\\
				& \lesssim \sum_{j =1}^{3} \left\| |\xi|^k \left( \tilde{\mathcal{G}}_{11}(\xi, t) \widehat{\eta_j^0} + \tilde{\mathcal{G}}_{12}(\xi, t) \widehat{u_j^0}\right) \right\|_0^2 \nonumber\\
				&\quad +\sum_{j =1}^{3} \left\|  |\xi|^k \left( \bar{\mathcal{G}}_{11}(\xi, t) \widehat{\eta_j^0} + \bar{\mathcal{G}}_{12}(\xi, t) \widehat{u_j^0}\right) \right\|_0^2 \nonumber\\
				& \lesssim (1 +t)^{-\frac{3}{2} -k}\left( \|(\eta^0, u^0)\|_{L^1}^2 + \|\nabla^k (\eta^0, u^0)\|_0^2 \right), \label{4.18}
			\end{align}
			where $0 \leqslant k \leqslant 2$. Similarly to \eqref{4.18}, we can deduce from \eqref{4.90} and \eqref{4.91} that 
			\begin{align}
				\|\nabla^k u^l \|_0^2 & \lesssim \left\| |\xi|^{-2} \xi \otimes \xi |\xi|^k \left( \tilde{\mathcal{G}}_{21}(\xi, t) \widehat{\eta^0} + \tilde{\mathcal{G}}_{22}(\xi, t) \widehat{u^0} \right) \right\|_0^2 \nonumber\\
				& \quad +\left\| \left( I - |\xi|^{-2} \xi \otimes \xi  \right) |\xi|^k \left( \bar{\mathcal{G}}_{21}(\xi, t) \widehat{\eta^0} + \bar{\mathcal{G}}_{22}(\xi, t) \widehat{u^0} \right) \right\|_0^2 \nonumber\\
				& \lesssim \sum_{j =1}^{3} \left\| |\xi|^k \left( \tilde{\mathcal{G}}_{21}(\xi, t) \widehat{\eta_j^0} + \tilde{\mathcal{G}}_{22}(\xi, t) \widehat{u_j^0}\right) \right\|_0^2 \nonumber\\
				&\quad +\sum_{j =1}^{3} \left\|  |\xi|^k \left( \bar{\mathcal{G}}_{21}(\xi, t) \widehat{\eta_j^0} + \bar{\mathcal{G}}_{22}(\xi, t) \widehat{u_j^0}\right) \right\|_0^2 \nonumber\\
				& \lesssim (1 +t)^{-\frac{3}{2} -k}\left( \|(\eta^0, u^0)\|_{L^1}^2 + \|\nabla^k (\eta^0, u^0)\|_0^2 \right),  \label{4.46}
			\end{align}
			for $0 \leqslant k \leqslant 2$. Therefore, combining  \eqref{4.18} and \eqref{4.46}, we can directly have \eqref{4.83} . 
		
	\end{proof}
	
	\subsection{Decay estimates for the nonlinear problem}\label{202410272204}
	This section is served to elucidate the long-time behaviour of solutions to the nonlinear Cauchy problem \eqref{1.17n}.
	Let $(\eta,u)$ be the solution given by Theorem \ref{1.12} and $\mathcal{N}$ be defined by \eqref{1.4}. Since \eqref{1.17n} and \eqref{1.21} are equivalent to each other, the problem \eqref{1.17n} can be rewritten as
	\begin{equation}\label{4.28}
		\begin{cases}
			U_t =G U +\mathcal{H}, \\
			U^0 = (\eta^0, u^0),
		\end{cases}
		\mbox{where }  U=
		\begin{bmatrix}
			\eta\\
			u
		\end{bmatrix} \mbox{ and }\mathcal{H} =
		\begin{bmatrix}
			0\\ \mathcal{N}
		\end{bmatrix}.
	\end{equation}
	According to the Duhamel principle, we can express the solution $U$  by
	\begin{align}
		\label{4.29}
		U = U^l +\int_0^t e^{(t -\tau)G} \mathcal{H}(\tau) \mm{d}\tau,
	\end{align}
	where $U^l=(\eta^l,u^l)$.
	Before the derivation of finer decay estimates of $U$,  we shall renew to finely estimate for the norms $\|\mathcal{N}^u\|_{L^1}$, $\|\nabla^k \mathcal{N}^u\|_0$, $\|\mathcal{N}_P\|_{L^1}$ and $\|\nabla^k \mathcal{N}_P\|_0$
	with $k =0$, $1$ and $2$. To this purpose, we define that
	\begin{align} 
		\mathcal{S}(t):= & \underset{0 \leqslant \tau \leqslant t, \, k=0, 1, 2}{\sup} \left\{ (1 +\tau)^{\frac{3}{4} +\frac{k}{2}}(\|\nabla^k \eta\|_0 +  \|\nabla^k u\|_0) +(1 + \tau)(\sqrt{\kappa}\|\Delta \nabla \eta\|_1+ \|\nabla^2 u\|_1) \right. \nonumber\\
		& \left.  +\sqrt{\kappa}\|\nabla \eta\|_2 +\sqrt{\kappa}(1+\tau)^{\frac{1}{2}} \|\Delta \eta\|_2 + \|(1 +\tau)\nabla \Delta u \|_{L^2(I_t; H^1)} \right\}. \nonumber
	\end{align}
	We then infer that the above norms of $\mathcal{N}^u$ can be upper bounded by $\mathcal{S}(t)$ as follows.
	\begin{lem}
		\label{4.13}
		Let  $\tau \in [0, t]$ with any given $t \geqslant 0$, the followings hold.
		\begin{align}
			& \|\mathcal{N}^u(\tau)\|_{L^1}  \lesssim \kappa^{-\frac{1}{2}} (1 +\tau)^{-\frac{7}{4}} \sqrt{\mathcal{E}^0}\mathcal{S}(t), \label{4.1} \\
			& \|\nabla^k\mathcal{N}^u(\tau) \|_0  \lesssim \kappa^{-\frac{1}{2}} (1 +\tau)^{-\frac{7}{4}}\sqrt{\mathcal{E}^0} \mathcal{S}(t)\mbox{ for } k=0\mbox{ and } 1, \label{4.2} \\
			& \|\nabla^2\mathcal{N}^u(\tau)\|_0  \lesssim \kappa^{-\frac{1}{2}}\sqrt{\mathcal{E}^0}\left( (1 +\tau)^{-\frac{7}{4}}\mathcal{S}(t) +  (1 +\tau)^{-\frac{3}{4}}  \|\nabla \Delta u\|_1\right).\label{4.3}
		\end{align}
	\end{lem}
	\begin{proof}
		Recalling  \eqref{0.11}, \eqref{2.7} and the definition of $\mathcal{N}^u $, we have
		\begin{align}
			\| \mathcal{N}^u(\tau) \|_{L^1} & \lesssim \|\nabla^2 \eta\|_0 \|\nabla u\|_0 +\|\nabla \eta\|_0 \|\nabla^2 u\|_0 \nonumber\\
			& \lesssim \kappa^{-\frac{1}{2}}(1 +\tau)^{-\frac{7}{4}}\left(   \sqrt{\kappa(1 +\tau)}\|\nabla^2 \eta\|_2 (1 +\tau)^\frac{5}{4} \|\nabla u\|_0  \right.\nonumber\\
			& \quad +\left. \sqrt{\kappa} \|\nabla \eta\|_{2} (1 +\tau)^\frac{7}{4}\|\nabla^2 u\|_0 \right)\nonumber\\
			& \lesssim \kappa^{-\frac{1}{2}}(1 +\tau)^{-\frac{7}{4}} \sqrt{\mathcal{E}_0} \mathcal{S}(t), \nonumber
		\end{align}
		which implies \eqref{4.1}.
		
		From the above similar arguments, it can be obtained that
		\begin{align}
			\|\nabla \mathcal{N}^u \|_0 \lesssim &\sum_{i,j,k,l=1}^3
			\left(\| \nabla^2 \eta \|_{L^4} \|\nabla(\partial_i\eta_j\partial_ku_l)\|_{L^4} +\|\nabla^2(\partial_i\eta_j\partial_ku_l)\|_0 \right)\nonumber \\
			\lesssim &\sum_{i,j,k,l=1}^3
			\left( \| \nabla^2 \eta \|_{L^4}\|\nabla(\partial_i\eta_j\partial_ku_l) \|_0+\|\nabla^2(\partial_i\eta_j\partial_ku_l) \|_0\right),
			\label{4.4}
		\end{align}
where, in the last inequality in \eqref{4.4}, we have used the smallness of $\|\nabla \eta\|_{3}$ as \eqref{2.7} and the estimate
			$$\|\nabla(\partial_i\eta_j\partial_ku_l) \|_{L^4} \lesssim \|\nabla(\partial_i\eta_j\partial_ku_l) \|_0 + \|\nabla^2(\partial_i\eta_j\partial_ku_l) \|_0.$$
		By H$\mathrm{\ddot{o}}$lder's and Sobolev's inequalities, we get
		\begin{align}
			\|\nabla^2 \eta\|_{L^4} &  \lesssim  \|\nabla^2 \eta\|_0^\frac{1}{4} \|\nabla^3 \eta\|_0^\frac{3}{4} \nonumber\\
			&  \lesssim \kappa^{-\frac{1}{2}} (1 +\tau)^{-\frac{7}{8}} \left(({\kappa (1 +\tau))^{1/2}} \|\nabla^2 \eta\|_2 \right)^{\frac{1}{4}}\left(\sqrt{\kappa} (1 +\tau) \|\nabla^3 \eta\|_1\right)^{\frac{3}{4}} \nonumber\\
			& \lesssim \kappa^{-\frac{1}{2}} (1 +\tau)^{-\frac{7}{8}} \sqrt{\mathcal{E}_0},\label{4.25}\\
			\| \nabla(\partial_i\eta_j\partial_ku_l)\|_0 & \lesssim \|\nabla^2 \eta \nabla u\|_0 +\|\nabla \eta \nabla^2 u\|_0 \nonumber\\
			& \lesssim \|\nabla^2 \eta\|_{L^4} \|\nabla u\|_{L^4} +\|\nabla \eta\|_{L^4} \|\nabla^2 u\|_{L^4} \nonumber\\
			& \lesssim \|\nabla \eta\|_2 \left(\|\nabla u\|_0^\frac{1}{4} \|\nabla^2 u\|_0^\frac{3}{4} +\|\nabla^2 u\|_0^\frac{1}{4}\|\nabla^3 u\|_0^\frac{3}{4}\right) \nonumber\\
			& \lesssim (1 +\tau)^{-\frac{13}{8}}\left((1 +\tau)^{{5}/{4} }\|\nabla u\|_0\right)^\frac{1}{4} \left((1 +\tau)^{7/4} \|\nabla^2 u\|_0\right)^\frac{3}{4} \nonumber\\
			& \quad  +(1 +\tau)^{-\frac{19}{16}}\left((1 +\tau)^{7/4} \|\nabla^2 u\|_0\right)^\frac{1}{4} \left((1 +\tau) \|\nabla^2 u\|_1 \right)^\frac{3}{4} \nonumber \\
			&\lesssim  (1 +\tau)^{-\frac{19}{16}} \mathcal{S}(t)
			\label{4.6}
		\end{align}
		and
		\begin{align}
			\|\nabla^2(\partial_i\eta_j\partial_ku_l)\|_0 & \lesssim \|\nabla^3 \eta \nabla u\|_0 +\|\nabla^2 \eta \nabla^2 u\|_0 +\|\nabla \eta \nabla^3 u\|_0 \nonumber\\
			& \lesssim \|\nabla^3 \eta\|_0 \|\nabla  u\|_{L^\infty} +\|\nabla^2 \eta\|_{L^4} \|\nabla^2  u\|_{L^4} +\|\nabla \eta\|_{L^\infty} \|\nabla^3  u\|_0 \nonumber\\
			& \lesssim \|\nabla^3 \eta\|_0 \|\nabla u\|_0^{\frac{1}{4}} \|\nabla^3 u\|_0^{\frac{3}{4}} + \|\nabla \eta\|_0^{\frac{1}{4}} \|\nabla^3 \eta\|_0^{\frac{3}{4}} \|\nabla^3 u\|_0 \nonumber\\
			& \quad +\|\nabla^2 \eta\|_0^\frac{1}{4} \|\nabla^3 \eta\|_0^\frac{3}{4} \|\nabla^2 u\|_0^{\frac{1}{4}} \|\nabla^3 u\|_0^{\frac{3}{4}}\nonumber\\
			& \lesssim \kappa^{-\frac{1}{2}}(1 +\tau)^{-\frac{33}{16}} \left((1 +\tau) \|\nabla^2 u\|_1 \right)^\frac{3}{4}\left(  {\kappa}^{1/2} (1 +\tau) \|\nabla^3 \eta\|_1\left((1 +\tau)^{{5/4}} \|\nabla u\|_0 \right)^\frac{1}{4}\right.\nonumber\\
			&\quad +  \left. \left((\kappa (1 +\tau))^{1/2} \|\nabla^2 \eta\|_2\right)^\frac{1}{4} \left({\kappa}^{1/2} (1+ \tau)\|\nabla^3 \eta\|_1\right)^\frac{3}{4}   \left((1 +\tau)^{7/4} \|\nabla^2 u\|_0 \right)^\frac{1}{4}   \right)\nonumber\\
			&\quad +\kappa^{-\frac{1}{2}} (1 +\tau)^{-\frac{7}{4}} \left({\kappa}^{1/2} \|\nabla \eta\|_2\right)^\frac{1}{4} \left({\kappa}^{1/2} (1+ \tau)\|\nabla^3 \eta\|_1\right)^\frac{3}{4} \left((1 +\tau)\|\nabla^2 u\|_1 \right)\nonumber\\
			& \lesssim \kappa^{-\frac{1}{2}} (1 +\tau)^{-\frac{7}{4}} \sqrt{\mathcal{E}_0} \mathcal{S}(t).
			\label{4.7}
		\end{align}
		Inserting \eqref{4.25}--\eqref{4.7} into \eqref{4.4} yields the estimate \eqref{4.2} for $k =1$. Besides, we have
		\begin{align}
			\| \mathcal{N}^u(\tau) \|_0 \lesssim \|\mathcal{N}^u(\tau)\|_{L^1}^{\frac{2}{5}} \|\nabla \mathcal{N}^u(\tau)\|_0^{\frac{3}{5}} \lesssim \kappa^{-\frac{1}{2}} (1 +\tau)^{-\frac{7}{4}}\sqrt{\mathcal{E}^0} \mathcal{S}(t), \nonumber
		\end{align}
		which leads to \eqref{4.2} for $k =0$.
		
		The norm $\|\nabla^2 \mathcal{N}^u\|_0$ will be addressed as follows.
		\begin{align}
			\|\nabla^2 \mathcal{N}^u\|_0 & \lesssim\sum_{i,j,k,l=1}^3\left( \|(\nabla^2 \eta)^2 \nabla (\partial_i\eta_j\partial_ku_l)\|_0 +\|\nabla^3 \eta\nabla (\partial_i\eta_j\partial_ku_l)\|_0\right. \nonumber \\
			&\qquad \left.  +\|\nabla^2 \eta \nabla^2 (\partial_i\eta_j\partial_ku_l) \|_0 +\|\nabla^3(\partial_i\eta_j\partial_ku_l)\|_0\right) \nonumber\\
			& \lesssim \sum_{i,j,k,l=1}^3\left(\|\nabla^2 \eta\|_{L^6}^2 \|\nabla (\partial_i\eta_j\partial_ku_l)\|_{L^6} +\|\nabla^3 \eta\|_{L^3} \|\nabla (\partial_i\eta_j\partial_ku_l) \|_{L^6} \right.\nonumber \\
			&\qquad \left.  +\|\nabla^2 \eta\|_{L^\infty} \|\nabla^2 (\partial_i\eta_j\partial_ku_l)\|_0 +\|\nabla^3(\partial_i\eta_j\partial_ku_l)\|_0 \right)\nonumber\\
			& \lesssim (1 +\|\nabla^3 \eta\|_1 ) \|\nabla^2(\partial_i\eta_j\partial_ku_l)\|_0 +\|\nabla^3(\partial_i\eta_j\partial_ku_l)\|_0 \nonumber\\
			& \lesssim \sum_{i,j,k,l=1}^3\left(\left(1 +\kappa^{-{1/2}}(1 +\tau)^{-1}\left(\kappa^{1/2}  (1 +\tau) \|\nabla^3 \eta\|_1 \right) \right)\|\nabla^2(\partial_i\eta_j\partial_ku_l)\|_0 \right.\nonumber \\
			&\qquad \left.+\|\nabla^3(\partial_i\eta_j\partial_ku_l)\|_0\right) \nonumber\\
			& \lesssim \sum_{i,j,k,l=1}^3\left(\left(1 +\kappa^{-{1/2}}(1 +\tau)^{-1} \sqrt{\mathcal{E}^0} \right)\|\nabla^2(\partial_i\eta_j\partial_ku_l)\|_0 +\|\nabla^3 (\partial_i\eta_j\partial_ku_l)\|_0\right).
			\label{4.10}
		\end{align}
		Note that the norm $\|\nabla^2 (\partial_i\eta_j\partial_ku_l)\|_0$ can be estimated in a manner that differs from \eqref{4.7} as follows:
		\begin{align}
			\|\nabla^2(\partial_i\eta_j\partial_ku_l)\|_0 & \lesssim  \|\nabla \eta\|_2 \|\nabla u\|_0^{\frac{1}{4}} \|\nabla^3 u\|_0^{\frac{3}{4}} + \|\nabla \eta\|_2 \|\nabla^3 u\|_0 \nonumber\\
			&\quad + \|\nabla \eta\|_2 \|\nabla^2 u\|_0^{\frac{1}{4}} \|\nabla^3 u\|_0^{\frac{3}{4}}  \lesssim (1 +t)^{-1} \mathcal{S}(t).
			\label{4.11}
		\end{align}
		In addition, we can obtain
		\begin{align}
			& \|\nabla^3(\partial_i\eta_j\partial_ku_l) \|_0 \nonumber\\
			& \lesssim \|\nabla^4 \eta\|_0 \|\nabla u\|_{L^\infty} +\|\nabla^3 \eta\|_{L^4} \|\nabla^2 u\|_{L^4} \nonumber\\
			&\quad +\|\nabla^2 \eta\|_{L^\infty} \|\nabla^3 u\|_0 +\|\nabla \eta\|_{L^\infty} \|\nabla^4 u\|_0 \nonumber\\
			& \lesssim \|\nabla^3 \eta\|_1 \left( \|\nabla^2 u\|_0^\frac{1}{2} \|\nabla^3 u\|_0^\frac{1}{2} + \|\nabla^2 u\|_0^\frac{1}{4} \|\nabla^3 u\|_0^\frac{3}{4} \right)\nonumber\\
			& \quad +\|\nabla^2 \eta\|_0^\frac{1}{4} \|\nabla^4 \eta\|_0^\frac{3}{4} \|\nabla^3 u\|_0 +\|\nabla^2 \eta\|_0^\frac{1}{2} \|\nabla^3 \eta\|_0^\frac{1}{2} \|\nabla^3 u\|_1\nonumber\\
			& \lesssim \kappa^{-\frac{1}{2}}\left( (1 +\tau)^{-\frac{19}{8}} \left(\sqrt{\kappa} (1 +\tau)\|\nabla^3 \eta\|_1\right) \left((1 +\tau)^{7/4} \|\nabla^2 u\|_0\right)^\frac{1}{2} \left((1 +\tau) \|\nabla^2 u\|_1\right)^\frac{1}{2}\right. \nonumber\\
			&\quad + (1 +\tau)^{-\frac{35}{16}}\left(\sqrt{\kappa} (1 +\tau)\|\nabla^3 \eta\|_1\right) \left((1 +\tau)^{7/4} \|\nabla^2 u\|_0\right)^\frac{1}{4} \left((1 +\tau) \|\nabla^2 u\|_1\right)^\frac{3}{4} \nonumber\\
			&\quad +\ (1 +\tau)^{-\frac{15}{8}} \left((\kappa (1 +\tau))^{1/2} \|\nabla^2 \eta\|_2 \right)^\frac{1}{4} \left(\sqrt{\kappa}(1 +\tau)\|\nabla^3 \eta\|_1 \right)^\frac{3}{4} \left((1 +\tau) \|\nabla^2 u\|_1\right) \nonumber\\
			&\quad + \left. (1 +\tau)^{-\frac{3}{4}} \left((\kappa (1 +\tau))^{1/2} \|\nabla^2 \eta\|_2\right)^\frac{1}{2} \left(\sqrt{\kappa} (1 +\tau) \|\nabla^3 \eta\|_1\right)^\frac{1}{2} \|\nabla^3 u\|_1 \right)\nonumber\\
			& \lesssim \kappa^{-\frac{1}{2}} \left((1 +\tau)^{-\frac{7}{4}} \sqrt{\mathcal{E}_0}\mathcal{S}(t) +(1 +\tau)^{-\frac{3}{4}} \sqrt{\mathcal{E}^0} \|\nabla \Delta u\|_1 \right). \label{4.12}
		\end{align}
		Then, the desired estimate \eqref{4.3} can be deduced by inserting \eqref{4.7}, \eqref{4.11} and \eqref{4.12} into \eqref{4.10}.
	\end{proof}
	We now proceed to estimate the norms associated with $\mathcal{N}_P$.
	\begin{lem}
		\label{4.21}
		Let  $\tau \in [0, t]$ with any given $t \geqslant 0$, we have\begin{align}
			\|\mathcal{N}_P(\tau)\|_{L^1}, \|\nabla^k \mathcal{N}_P(\tau)\|_0 \lesssim \kappa^{-\frac{1}{2}}(1 +\tau)^{-\frac{7}{4}}\sqrt{\mathcal{E}^0}\mathcal{S}(t),
			\label{4.14}
		\end{align}
		where $k=0$, $1$ and $2$.
	\end{lem}
	\begin{proof}
		By \eqref{0.11} and Lemma \ref{4.58}, we have
		\begin{align}
			\|\mathcal{N}_P(\tau)\|_{L^1} & \lesssim \||\nabla \eta ||\nabla^2 \eta|\|_{L^1} \lesssim \|\nabla \eta\|_0 \|\nabla^2 \eta\|_0 \nonumber\\
			& \lesssim \kappa^{-\frac{1}{2}} (1 +\tau)^{-\frac{7}{4}}  \sqrt{\mathcal{E}^0}\mathcal{S}(t).
			\label{4.15}
		\end{align}
		Using H$\mathrm{\ddot{o}}$lder's and Sobolev's inequalities, it can be derived that
		\begin{align}
			\|\nabla^2 \mathcal{N}_P(\tau) \|_0 & \lesssim \||\nabla^2 \eta|^3\|_0 +\||\nabla^3 \eta | |\nabla^2 \eta|\|_0 +\||\nabla \eta ||\nabla^4 \eta|\|_0 \nonumber\\
			& \lesssim \|\nabla^2 \eta\|_{L^6}^3 +\|\nabla^3 \eta\|_{L^3} \|\nabla^2 \eta\|_{L^6} +\|\nabla \eta\|_{L^\infty} \|\nabla^4 \eta\|_0 \nonumber\\
			& \lesssim \|\nabla^3 \eta\|_1^2 +\|\nabla^2 \eta\|_0^\frac{1}{2} \|\nabla^3 \eta\|_0^\frac{1}{2} \|\nabla^3 \eta\|_1 \nonumber\\
			&\lesssim \kappa^{-1}(1 +\tau)^{-2} \left(\sqrt{\kappa}(1 +\tau) \|\nabla^3 \eta\|_1\right)^2 \nonumber\\
			&\quad +\kappa^{-\frac{3}{4}} (1 +\tau)^{-\frac{19}{8}}\left((1 +\tau)^{7/4} \|\nabla^2 \eta\|_0\right)^\frac{1}{2} \left(\sqrt{\kappa}(1 +\tau) \|\nabla^3 \eta\|_1 \right)^\frac{3}{2} \nonumber\\
			& \lesssim \kappa^{-\frac{1}{2}} (1 +\tau)^{-\frac{7}{4}} \sqrt{\mathcal{E}^0} \mathcal{S}(t),
			\label{4.16}
		\end{align}
		which implies the estimate \eqref{4.14} for $k =2$. By  \eqref{4.15} and \eqref{4.16}, we get
		\begin{align}
			& \|\mathcal{N}_P(\tau)\|_0  \lesssim \|\mathcal{N}_P\|_{L^1}^\frac{4}{7} \|\nabla^2 \mathcal{N}_P\|_0^\frac{3}{7} \lesssim \kappa^{-\frac{1}{2}} (1 +\tau)^{-\frac{7}{4}}\sqrt{\mathcal{E}^0}\mathcal{S}(t), \nonumber\\
			& \|\nabla \mathcal{N}_P(\tau)\|_0  \lesssim \| \mathcal{N}_P\|_{L^1}^\frac{2}{7} \|\nabla^2 \mathcal{N}_P\|_0^\frac{5}{7} \lesssim \kappa^{-\frac{1}{2}} (1 +\tau)^{-\frac{7}{4}}\sqrt{\mathcal{E}^0}\mathcal{S}(t), \nonumber
		\end{align}
		which lead to \eqref{4.14} for $k =0$ and $1$.
	\end{proof}
	
	\emph{Now we are in the position for the proof of  Theorem \ref{1.24}}.  
		To this end, it suffices to prove that for any given $T \geqslant 0$, there exists a positive constant $c_3$, independent of $T$,  such that
	\begin{align}
		\mathcal{S}(T) \leqslant c_3. \label{4.50}
	\end{align}
	
Let us first recall a well-known result  (see e.g., \cite[Lemma 2.5] {MR2325837} for the proof):
		let the real numbers $r_1 >0$ and $r_2 \in [0, r_1]$, then there is a positive constant $\tilde{c}$, depending only on $r_1$ and $r_2$, such that
		\begin{align}
			\int_0^t (1 +t -\tau)^{-r_1} (1 +\tau)^{-r_2} \mm{d}\tau \leqslant \tilde{c}( 1 +t)^{-r_2}
			\mbox{ for any }t \geqslant 0.  \label{4.74}
	\end{align}
Thus, for any given $t \in \overline{I_T}$, by Proposition \ref{4.8}, Lemmas \ref{4.13}--\ref{4.21} and \eqref{4.74}, we can derive from \eqref{4.29} that
	\begin{align}
		\| {(\eta, u)}\|_0 & \lesssim  \|  {(\eta^l, u^l)} \|_0 +\int_0^t \|  {e^{(t -\tau)G} \mathcal{H}(\tau)}\|_0 \mathrm{d}\tau \nonumber\\
		& \lesssim (1 +t)^{-\frac{3}{4}} \left( \|(\eta^0, u^0)\|_{L^1} +\| (\eta^0, u^0)\|_0 \right) \nonumber\\
		&\quad +\int_0^t (1+ t-\tau)^{-\frac{3}{4}} \left( \|\mathcal{N}(\tau)\|_{L^1} +\| \mathcal{N}(\tau)\|_0 \right) \mathrm{d}\tau \nonumber\\
		& \lesssim (1 +t)^{-\frac{3}{4}} \left( \|(\eta^0, u^0)\|_{L^1} +\| (\eta^0, u^0)\|_0 \right) \nonumber\\
		& \quad +\int_0^t (1 +t -\tau)^{-\frac{3}{4}} \left(\|(\mathcal{N}^u(\tau),\mathcal{N}_P(\tau))\|_{L^1}+\|(\mathcal{N}^u(\tau),\mathcal{N}_P(\tau))\|_0\right) \mathrm{d}\tau \nonumber\\
		& \lesssim  (1 +t)^{-\frac{3}{4}} \left( \|(\eta^0, u^0)\|_{L^1} +\| (\eta^0, u^0)\|_0 \right) \nonumber\\
		&\quad +\kappa^{-\frac{1}{2}}\sqrt{\mathcal{E}^0} \int_0^t (1 +t -\tau)^{-\frac{3}{4}} (1 +\tau)^{-\frac{7}{4}} \mathcal{S}(t) \mathrm{d}\tau \nonumber\\
		&\lesssim  (1 +t)^{-\frac{3}{4}}  \left( \|(\eta^0, u^0)\|_{L^1} +\| (\eta^0, u^0)\|_0 +\kappa^{-\frac{1}{2}}\sqrt{\mathcal{E}^0}\mathcal{S}(T) \right) ,
		\label{4.17}
		\\
		\| {\nabla (\eta, u)} \|_0 & \lesssim \| {\nabla (\eta^l, u^l)} \|_0 +\int_0^t \| {\nabla( e^{(t -\tau)G} \mathcal{H}(\tau)})\|_0 \mathrm{d}\tau \nonumber\\
		& \lesssim (1 +t)^{-\frac{5}{4}} \left( \|(\eta^0, u^0)\|_{L^1} +\|\nabla (\eta^0, u^0)\|_0 \right) \nonumber\\
		&\quad +\int_0^t (1+ t-\tau)^{-\frac{5}{4}} \left( \|\mathcal{N}(\tau)\|_{L^1} +\|\nabla \mathcal{N}(\tau)\|_0 \right) \mathrm{d}\tau \nonumber\\
		& \lesssim (1 +t)^{-\frac{5}{4}} \left( \|(\eta^0, u^0)\|_{L^1} +\|\nabla (\eta^0, u^0)\|_0 \right) \nonumber\\
		& \quad +\int_0^t (1 +t -\tau)^{-\frac{5}{4}} \left(\|(\mathcal{N}^u(\tau),\mathcal{N}_P(\tau))\|_{L^1}+
		\|\nabla (\mathcal{N}^u(\tau), \mathcal{N}_
		P(\tau))\|_0 \right) \mathrm{d}\tau \nonumber\\
		& \lesssim  (1 +t)^{-\frac{5}{4}} \left( \|(\eta^0, u^0)\|_{L^1} +\|\nabla (\eta^0, u^0)\|_0 \right) \nonumber\\
		&\quad +\kappa^{-\frac{1}{2}}\sqrt{\mathcal{E}^0} \int_0^t (1 +t -\tau)^{-\frac{5}{4}} (1 +\tau)^{-\frac{7}{4}}\mathcal{S}(t) \mathrm{d}\tau \nonumber\\
		&\lesssim  (1 +t)^{-\frac{5}{4}} \left( \|(\eta^0, u^0)\|_{L^1} +\|\nabla (\eta^0, u^0)\|_0 +\kappa^{-\frac{1}{2}}\sqrt{\mathcal{E}_0}\mathcal{S}(T)\right) .
		\label{4.22}
	\end{align}
	Similarly, it holds that
	\begin{align}
		\| {\nabla^2 (\eta, u)}\|_0 & \lesssim \| {\nabla^2 (\eta^l, u^l)} \|_0 +\int_0^t \| {\nabla^2 (e^{(t -\tau) G} \mathcal{H}(\tau)})\|_0 \mathrm{d}\tau \nonumber\\
		& \lesssim (1 +t)^{-\frac{7}{4}} \left( \|(\eta^0, u^0)\|_{L^1} +\|\nabla^2(\eta^0, u^0)\|_0 \right) \nonumber\\
		& \quad +\int_0^t (1 +t -\tau)^{-\frac{7}{4}} \big(\|(\mathcal{N}^u(\tau),\mathcal{N}_P(\tau))\|_{L^1}+
		\|\nabla^2 (\mathcal{N}^u(\tau),\mathcal{N}_P(\tau))\|_0 \big) \mathrm{d}\tau \nonumber\\
		& \lesssim (1 +t)^{-\frac{7}{4}} \left( \|(\eta^0, u^0)\|_{L^1} +\|\nabla^2(\eta^0, u^0)\|_0 \right) \nonumber\\
		& \quad +\kappa^{-\frac{1}{2}} \mathcal{S}(t)\int_0^t (1 +t-\tau)^{-\frac{7}{4}}
		\left( (1 +\tau)^{-\frac{7}{4}} +
		(1 +\tau)^{-\frac{3}{4}} \|\nabla \Delta u\|_1   \right)\mathrm{d}\tau \nonumber\\
		& \lesssim (1 +t)^{-\frac{7}{4}} \big( \|(\eta^0, u^0)\|_{L^1} +\|\nabla^2(\eta^0, u^0)\|_0 \nonumber \\
		&\quad +\kappa^{-\frac{1}{2}} \mathcal{S}(t)(\sqrt{\mathcal{E}^0} + \kappa^{-\frac{1}{2}}
		\|(1  +\tau)\nabla \Delta u\|_{L^2(I_t; H^1)})\big)\nonumber\\
		& \lesssim (1 +t)^{-\frac{7}{4}}\left( \|(\eta^0, u^0)\|_{L^1} +\|\nabla^2(\eta^0, u^0)\|_0 +\kappa^{-\frac{1}{2}}\sqrt{\mathcal{E}^0} \mathcal{S}(T)\right).
		\label{4.23}
	\end{align}
	
	Consequently, we immediately derive from \eqref{0.11} and  {\eqref{4.17}}--\eqref{4.23} that
	\begin{align}
		\mathcal{S}(T) \leqslant c_4 \left( \|(\eta^0, u^0)\|_{L^1} +
		\sqrt{\mathcal{E}^0} \right) + c_4 \kappa^{-\frac{1}{2}} \sqrt{\mathcal{E}^0} \mathcal{S}(T),
		\label{4.26}
	\end{align}
	where $c_4$ is a positive constant independent on $T$.
	In addition, by Young's inequality and \eqref{1.22}, we have
	$$\kappa^{-\frac{1}{2}}\sqrt{\mathcal{E}^0} \lesssim \kappa^{-\frac{1}{2}} \max\left\{ (2 c_2
	\mathcal{E}^0)^\frac{1}{4}, 2 c_2\mathcal{E}^0 \right\} \lesssim \sqrt{c_1}
	.$$
Then,  from the smallness of $c_1$, the above estimate and \eqref{4.26}, we can obtain \eqref{4.50}  with 
		$$c_3: = 2c_4 \left(  \|(\eta^0, u^0)\|_{L^1} + \sqrt{\mathcal{E}^0} \right).$$ The desired decay estimate \eqref{1.26} follows from the boundedness of $\mathcal{S}(T)$. This completes  the proof of  Theorem \ref{1.24}.
	
	\vspace{4mm} \noindent\textbf{Acknowledgements.}
	The research of Fei Jiang was supported by NSFC (Grant Nos. 12022102 and 12231016), and the Natural Science Foundation of Fujian Province of China (Nos. 2022J01105 and 2024J01610017) and the Central Guidance on Local Science and Technology Development Fund of Fujian Province.
	\renewcommand\refname{References}
	\renewenvironment{thebibliography}[1]{%
		\section*{\refname}
		\list{{\arabic{enumi}}}{\def\makelabel##1{\hss{##1}}\topsep=0mm
			\parsep=0mm
			\partopsep=0mm\itemsep=0mm
			\labelsep=1ex\itemindent=0mm
			\settowidth\labelwidth{\small[#1]}%
			\leftmargin\labelwidth \advance\leftmargin\labelsep
			\advance\leftmargin -\itemindent
			\usecounter{enumi}}\small
		\def\newblock{\ }
		\sloppy\clubpenalty4000\widowpenalty4000
		\sfcode`\.=1000\relax}{\endlist}
	\bibliographystyle{model1b-num-names}
	\bibliography{refs}
\end{CJK*}
\end{document}